\documentclass{amsart}%
\usepackage[T1]{fontenc}
\usepackage{amsfonts}
\usepackage{svg}
\usepackage{mathtools,graphicx}
\usepackage{longtable}
\usepackage{array}
\usepackage{latexsym}
\usepackage{hyperref}
\usepackage{amsmath}
\usepackage{mathrsfs}
\usepackage{pstricks}
\usepackage{listings}
\usepackage{hyperref}
\usepackage{enumitem}
\usepackage{csvsimple}
\usepackage{todonotes}
\usepackage{amssymb,bm}
\usepackage{amsmath}
\usepackage{amsthm}
\usepackage{tikz}
\usepackage{tikz-cd}
\usepackage{epsfig}
\usepackage{verbatim}
\usepackage{float}
\usepackage{tabularx}
\usepackage{algorithmic}
\usepackage{amssymb}
\usepackage{graphicx}%
\providecommand{\U}[1]{\protect\rule{.1in}{.1in}}
\definecolor{codegray}{rgb}{0.5,0.5,0.5}
\lstdefinestyle{mystyle}{
commentstyle=\color{codegreen},
keywordstyle=\color{magenta},
numberstyle=\tiny\color{codegray},
stringstyle=\color{codepurple},
basicstyle=\ttfamily\footnotesize,
breakatwhitespace=false,
breaklines=true,
captionpos=b,
keepspaces=true,
numbers=left,
numbersep=5pt,
showspaces=false,
showstringspaces=false,
showtabs=false,
tabsize=2
}
\lstdefinelanguage{GAP}{
basicstyle=\ttfamily,
keywords={true, false, function, return, fail, if, in, while, do, od, else, elif, fi, break, continue},
keywordstyle=\color{blue}\bfseries,
otherkeywords={      >, <, ==
},
breaklines=true,
identifierstyle=\color{black},
sensitive=True,
comment=[l]{\#},
commentstyle=\color{cyan},
stringstyle=\color{red},
morestring=[b]',
morestring=[b]"
}
\providecommand{\U}[1]{\protect\rule{.1in}{.1in}}

\newcolumntype{Y}{>{\raggedleft\arraybackslash}X}

\def\vs{\vskip.3cm}
\def\noi{\noindent}

\newtheorem{theorem}{Theorem}

\newtheorem{corollary}[theorem]{Corollary}

\newtheorem{definition}[theorem]{Definition}

\newtheorem{lemma}[theorem]{Lemma}

\newtheorem{proposition}[theorem]{Proposition}

\def\vs{\vskip.3cm}

\def\vs{\vskip.3cm}
\def\noi{\noindent}

\def\cV{\mathcal V}

\def\id{\text{\rm Id}}

\begin{document}
\title[Nonlinear Dynamics of Leaky ReLU Network]{Nonlinear Dynamics In Optimization Landscape of Shallow Neural Networks with Tunable Leaky ReLU}
\author[J. Liu]{Jingzhou Liu}
\address{Department of Mathematical Science The University of Texas at Dallas, 75080 USA}
%\address{Epimorphism. LLC}
\email{jingzhou.liu@tcu.edu}

\begin{abstract}
In this work, we study the nonlinear dynamics of a shallow neural network with leaky ReLU activation. Assuming equal input and hidden-layer width $k,$ we establish a unified theoretical framework, applicable to \textbf{arbitrary width} $\mathbf{k\ge 4}$, to detect branches of nontrivial solutions bifurcating from global minimum as leaky parameter $\alpha$ varies and to characterize their associated symmetries. The analysis is based on the methodology of equivariant gradient degree and provides a complete
topological classification of critical points for arbitrary $k$. We also show that such bifurcations are width-independent, arise only for nonnegative $\alpha$ and that the global minimum undergoes no further symmetry-breaking in the engineering regime $\alpha\in (0, 1).$ To illustrate these general results, an explicit example with $k=5$ is presented to exhibit the resulting bifurcation structure together with the symmetries of the corresponding solution branches.
\end{abstract}
\maketitle
\noi \textbf{Mathematics Subject Classification:} 68T07, 37G40, 47J15, 47H11,  

\medskip

\noi \textbf{Key Words and Phrases:} Nonlinear Dynamics; Neural Networks; Equivariant Degree; Symmetry; Bifurcation; Optimization Landscape.

\section{Introduction}
The optimization landscape of neural networks exhibits a rich structure shaped by high-dimensional nonconvexity and, in many cases, intrinsic symmetries. Shallow networks, particularly two-layer teacher–student architectures, have been widely studied as a canonical framework for understanding such optimization landscapes, as they provide a simplified yet representative setting amenable to rigorous theoretical analysis \cite{akiyama2021learnability,arjevani2025geometry, du2018gradient, soltanolkotabi2018theoretical}. In the teacher–student setting, a fixed teacher network serves as the ground-truth model, while a student network is trained to approximate the teacher's output by minimizing a prescribed loss function.
A prototypical example is the two-layer fully connected network with ReLU activation, whose loss landscape may exhibit numerous spurious local minima. Due to neuron-permutation symmetries, these minima can be organized into families of symmetry-related critical points, for which explicit analytical expressions are available \cite{symmetry2021, safr}. Recent studies further show that, as the number of neurons $k$ varies, certain families approach zero loss as $k$ increases, while others collapse into saddle points \cite{arjevani2021analytic, arjevani2022annihilation}. Related work, such as \cite{safran2021effects}, has further demonstrated that the loss is locally strongly convex around each global minimum, while spurious local minima may nevertheless persist. Moreover, even mild over-parameterization can turn such local minima into saddle points, making gradient-based methods more likely to reach a global minimum. Subsequent studies have extended these landscape analyses to deeper architectures \cite{arjevani2021symmetry}. Despite the rich analytical understanding of critical points of ReLU networks in the static setting, considerably less is known about how the loss landscape evolves under its natural extension to Leaky ReLU as the leaky parameter \(\alpha\) varies. Understanding such dynamics is particularly important, as variations in \(\alpha\) may induce topological transitions in the loss landscape and thereby affect the training dynamics. In this work, we employ the topological method of equivariant gradient degree to characterize such behaviors.
\vskip0.3cm

\noi By following the setting of \cite{ symmetry2021}, we consider a two-layer teacher student network trained under the mean-squared loss, where both the input and hidden layers have width $k$, and the teacher model is given by vectorized identity matrix. More precisely, let $x\in \mathbb{R}^k$ be the input of the neural network, where $x$ is sampled from a Gaussian distribution 
$\mathcal N(0,\mathbb I_k)$ and the leaky ReLU activation $\sigma_{\alpha}:\mathbb{R}\to\mathbb{R}$ given by 
\[
\sigma_{\alpha}(a)=\max\{(1-\alpha)a,a\},\;\;\;\alpha\in \mathbb R.
\] 
\noi It is worth noting that in practical engineering regime, $\alpha$ is typically chosen within the interval $(0,1)$ due to its empirical performance and reduces to linear activation when $\alpha=0$. \vskip0.3cm 
Consider a student network with a single hidden layer of $k$ neurons, denoted by
\[
u=(u_1,u_2,\cdots, u_k)^T,
\]
where each $u_i \in \mathbb{R}^k$ for $\;i\in \{1,\cdots,k\} $ represents the linear functional applied to the input $x\in \mathbb{R}^k$ in the $i$-th neuron. Let $v^o$ be the pre-trained weights in the teacher network %called \textit{target vector} 
and takes the form
\[
v^o = \sum_{i=1}^{k} e_i \otimes e_i \in \mathbb{R}^{k^2},
\]
where $e_i$ denotes the $i$-th standard basis vector in $\mathbb{R}^k$. Then the optimal solution for $u$ is obtained by minimizing the MSE loss function $\mathcal F_{\alpha}:\mathbb{R}^{k^2}\to \mathbb R,$
\begin{equation}\label{eq:Loss}
    \mathcal{F}_{\alpha}(u):=\frac{1}{2}\mathbb{E}_{x\sim \mathcal N(0,\mathbb{I}_k)}\Big(\sum_{i=1}^{k}\sigma_{\alpha}({u_i}^T x)-\sum_{i=1}^{ k}\sigma_{\alpha}({v_i}^T x) \Big)^2,
\end{equation}
which can be explicitly represented as
\begin{equation} \label{eq:explicit_l}
   \mathcal{F}_{\alpha}(u)=\sum_{i,j=1}^{k}\Big(\frac{1}{2}f_{\alpha}(u_i, u_j)-f_{\alpha}(u_i, v_j)+\frac{1}{2}f_{\alpha}(v_i,v_j)\Big),
\end{equation}
where $f_{\alpha}:\mathbb R^k\times \mathbb R^k \to \mathbb R$ is given by
    \begin{align} \label{eq: explicit_f}
        f_{\alpha}(w,v)=\frac{1}{2\pi}\|w\|\|v\|\Big(\alpha^2(\sin\theta-\theta\cos\theta)+(2+\alpha^2-2\alpha)\pi\cos\theta\Big), \;\; \theta=\cos^{-1}\frac
        {w\cdot v}{\|w\|\| v\|}.
    \end{align}
(One is referred to \cite{symmetry2021} Proposition 4.3 for results and direct derivation of equation \eqref{eq:explicit_l} and \eqref{eq: explicit_f}, and to \ref{app:explicit_form} for supplementary details.)\vskip0.3cm
In the setting considered, the system possesses intrinsic symmetries.
More precisely, on the space $\mathbb{R}^{k^2},$ one can define an orthogonal action of the group
\[
G:=S_k\times S_k,
\]
where the first $S_k$ acts by permuting the components $u_i$ while second $S_k$ denotes permutation within each $u_i.$ Explicitly, for $(\sigma,\gamma)\in G,$ the action of $G$ on $\mathbb R^{k^2}$ is given by
\[
(\sigma,\gamma)(u_1,u_2,\cdots,u_k)^T=(
\gamma u_{\sigma(1)}, \gamma u_{\sigma(2)}, \cdots, \gamma u_{\sigma(k)})^T.
\]
It is easy to observe that $\mathcal F_{\alpha}$ is $G$-invariant (see \cite{symmetry2021} Lemma 4.2-Example 4.8 for more details on the proof).\vskip0.3cm 
Let 
\[
\Omega:=\{u\in \mathbb R^{k^2}:u_i\neq 0, i=1,…,k\},
\]
Then $\nabla_{u}\mathcal F_{\alpha}$ is differentiable on $\Omega$ (see \cite{symmetry2021} Lemma 4.9) and notice that equation \eqref
{eq: system} admits trivial solution $v^o$, which represents the global minima of $\mathcal F_{\alpha}.$
The purpose of this work is to discuss solutions to 
\begin{equation}\label{eq: system}
\nabla_{u}\mathcal{F}_{\alpha}(u)=0.
\end{equation}
More precisely, by taking into account the symmetry $G$, we exam the branches of critical points of $\mathcal F_{\alpha}(u)$ and their symmetries emerging from the target vector $v^o$ as $\alpha$ varies.  
\vskip 0.3cm
In this work, we employ the equivariant gradient degree, originally introduced by K. Gęba \cite{gkeba1997degree}, as a tool to locate critical points of \eqref{eq: system} in a neighborhood of the orbit of global minima. This theoretic framework generalizes the classical Brouwer and Leray-Schauder degrees to gradient maps respecting group symmetries, and has been applied in a variety of symmetric variational problems. 
For completeness, we summarize the core ideas below.\vskip0.3cm
\noi Let $G$ be a compact Lie group and $\varphi_{\tilde{\alpha}}$ a $G$-invariant map.  Consider a neighborhood $\mathcal U$ of $G$-orbit of global minima $v^0$, the corresponding $G$-equivariant gradient degree is denoted by
$\nabla_{G}\text{\textrm{-deg}}\Big(\nabla \varphi_{\tilde{\alpha}}
,\mathcal U\Big)$ and takes values in the Euler ring $U(G)=\mathbb Z[\Phi(G)],$ where 
$\Phi(G)$ represents the set of conjugacy classes $(H)$ of closed subgroups $H\leq G$. Hence $\mathbb Z[\Phi(G)]$ is the free $\mathbb Z$-module generated by these orbit types. Explicitly, the degree can be expressed as
$\nabla_{G}\text{\textrm{-deg}}\Big(\nabla \varphi_{\tilde{\alpha}}
,\mathcal U\Big)=n_1(H_1)+n_2(H_2)+\cdots n_k(H_k),\;\; n_i\in \mathbb Z$
where $(H_i)$ corresponds to an orbit type in $\mathcal U.$ At a critical value $\tilde{\alpha}_o,$ the associated equivariant topological invariant  is defined by
\[
\omega_G(\tilde\alpha_0)
:= \nabla_G\text{-deg}\!\big(\nabla \varphi_{(\tilde\alpha_0)_-},\, \mathcal U\big)
 - \nabla_G\text{-deg}\!\big(\nabla \varphi_{(\tilde\alpha_0)_+},\, \mathcal U\big),
\]
and can be represented as
$\omega_G(\tilde{\alpha}_o)=r_1(K_1)+r_2(K_2)+\cdots +r_m(K_m),\;\; r_i\in \mathbb Z.$ 
Each nonzero coefficient $r_j$ indicates the emergence of a family of solutions bifurcating from $v^o,$ with symmetry at least $(K_j).$ This invariant thus provides a complete classification of bifurcating solutions as $\tilde{\alpha}$ passes through $\tilde{\alpha}_o.$ It is worthy to note that our approach offers an alternative to other tools for symmetric bifurcation, such as the equivariant singularity method, Lyapunov–Schmidt reduction, or center manifold theory.
Moreover, it is closely related to several other degree theories, primary degree, twisted degree, etc., for example. One can refer to \cite{balanov2025degree, balanov2006applied, hooton2017noninvasive, liu2023existence, eze2022subharmonic, balanov2010periodic, ize2003equivariant,ruan2008applications, dabkowski2017twisted,garcia2011global} for more details of those degrees and  applications. The essential properties of the equivariant gradient degree employed in this work are summarized in \ref{app:C}.\vskip0.3cm 
Our main result, obtained through the application of the equivariant gradient degree, is stated in Theorem \eqref{thm:main}. 
It shows that: 
\begin{itemize}
\item For any width $k \ge 4$ of the input and hidden layers, \textit{the system consistently undergoes bifurcations at three critical numbers, and their symmetries are associated with maximal orbit types in the following four $S_k$ irreducible representations: $S^{(k)}$, $S^{(k-1,1)}$, $S^{(k-2,2)}$, $S^{(k-2,1,1)}$, i.e. the trivial, standard, symmetric square and exterior square representation, respectively.}\textit{ In particular, at critical number $0$, the zero eigenvalue occurs simultaneously across three isotypic components, leading to multi-mode degeneracy and richer bifurcation structures.}\textit{For a concrete example when $k=5,$ there are at least four distinct symmetry types of maximal orbit kinds associated with bifurcating branches detected. \item The bifurcation occurs exclusively for nonnegative $\alpha$, and the critical numbers are independent of the network width $k$.  
\item Both the nonzero critical numbers converge to $2$ as $k$ goes to infinity, indicating %a width–invariant and asymptotically universal mechanism governing symmetry breaking in networks of our setting. Meanwhile, 
that throughout the practically relevant regime \(\alpha\in(0,1)\), the loss landscape undergoes no topological transition, and hence no nontrivial bifurcating branches emerge.} 
\end{itemize} 
\vskip0.3cm

\noindent The remainder of the paper is structured as follows. Section \ref{sec: framework} introduces the mathematical model for fully-connected two-layer teacher-student neural network. Subsection \ref{sec：2.1} restates the explicit form of loss function and its gradient for future use. Subsection \ref{sec:2.2} analyzes the isotropy group $\triangle S_k$ of global minima and the general $S_k$ isotypic decomposition of $\mathbb R^{k^2}$ for any number of neurons $k\ge 4.$ In Section \ref{sec:3}, we restate the general form of Hessian (Section \ref{sec:3.1}) and compute its spectrum at $v^o$ (Section \ref{sec:3.2}). The theoretical computation of gradient degree, including the main result \eqref{thm:main} and its proof, are presented in Section \ref{sec:4}. We then show a concrete example where $k=5$ in Section \ref{sec:5}. For the reader’s convenience, we also collect the derivations of the loss and its gradient, as well as the properties of the Euler ring and the equivariant gradient degree in Appendices \ref{app:A}, \ref{app:A2}, \ref{app:B}, and \ref{app:C}, respectively.  \vskip0.3cm
\section{Mathematical Framework}\label{sec: framework}
\subsection{Loss Function and Its Gradient.}\label{sec：2.1}
Let $V:=\mathbb R^{k^2}$ and consider k neurons $u:=(u_1,u_2,\cdots,u_k)^T\in V$ where $u_i\in \mathbb R^k,\; i=1,\cdots,k$ 
and
$\Omega:=\{u\in V:u_{i}
\not =0\}.$
The loss function $\mathcal F_{\alpha}:\Omega\to \mathbb R$ takes the form in \eqref{eq:explicit_l} and \eqref{eq: explicit_f}. The goal of this work is to explore solutions to 
$\nabla_u \mathcal F_{\alpha}(u)=0.$
The explicit form of gradient $\nabla_u \mathcal F_{\alpha}: \Omega\to \mathbb R^{k^2}$ is given by

\begin{equation}\label{eq:lg}
\nabla_{u}\mathcal F_{\alpha}(u)=
{\begin{bmatrix}
   \nabla_{u_1}\mathcal 
 F_{\alpha}(u),\nabla_{u_2}\mathcal F_{\alpha}(u),\cdots, \nabla_{u_k}\mathcal F_{\alpha}(u) 
\end{bmatrix}}^T,
\end{equation}
where %$\nabla_{\mathbf w^i}\mathcal F: \Omega \to \mathbb{R}^d$ and
\begin{align}\label{eq:explicit_lg}
\nabla_{u_i}\mathcal F(u)=\frac{\alpha}{2\pi}\sum_{j=1}^{k}\Big(\frac{\|u_j\|\sin\theta_{ij}}{\|u_i\|}u_i-\theta_{ij}u_j\Big)
-\frac{\alpha}{2\pi}\sum_{j=1}^{k}\Big(\frac{\sin\tilde{\theta}_{ij}}{\|u_i\|}u_i-\tilde{\theta}_{ij}v_j\Big)+\frac{1}{2}\sum_{j=1}^k\Big(u_j-v_j\Big), 
\end{align}
$({\theta}_{ij}=\cos^{-1}\frac{u_i\cdot u_j}{\|u_i\|\|u_j\|}\text{ and } \tilde{\theta}_{ij}=\cos^{-1}\frac{u_i\cdot v_j}{\|u_i\|\|v_j\|})$\vskip0.3cm
\noi see Appendix \eqref{ap:derivation_gls} and \cite{symmetry2021} Proposition 4.11 for more details about the derivation of equation \eqref{eq:explicit_lg}.\vskip0.3cm
\subsection{Symmetries and Isotypic Decomposition.} \label{sec:2.2} 
Recall that \(V=\mathbb{R}^{k^2}\) is a representation of
$G:=S_k\times S_k,$
with action
\[
(\sigma,\gamma)\,(u_1,\dots,u_k)^\top
=\big(\,\gamma\,u_{\sigma(1)},\,\gamma\,u_{\sigma(2)},\,\dots,\,\gamma\,u_{\sigma(k)}\big)^\top,
\]
where each  \(u_i\in\mathbb{R}^k\).  
Let \(v^o\) be the global minima of \eqref{eq: system} and its isotropy group is  given by
\[
G_{v^o}=\triangle S_k:=\{(\sigma,\sigma): \sigma\in S_k\}.
\]
Moreover, let $G(v^o)$ be  the $G$ orbit of $v^o,$ and its tangent space, denoted by $T_{v^o}G(v^o)$, can be obtained from the discreteness of the group $G.$ i.e. $T_{v^o}G(v^o)=\{0\}.$ Therefore, the slice at $v^o$ takes the following form \[ S_o:=\{u\in V: u\cdot T_{v^o}G(v^o)=0\}=V. \] Since $v^o$ has isotropy $\triangle S_k\cong S_k$, $S_o$ is a $S_k$ orthogonal representation. The purpose of the following is to obtain the general $S_k$ isotypic decomposition of the slice $S_o$ for any $k\in \mathbb N^+,\; k\ge2$.\vskip0.3cm
Let \(V_o:=\operatorname{span}\{\mathbf 1\},\) where $\mathbf 1=(1,\cdots,1)^T\in \mathbb R^k$ denotes the trivial $S_k$ representation and \(V_\perp:=\{x\in\mathbb R^k:\mathbf 1^\top x=0\}\) the standard $S_k$ representation. Then for any integer $k\ge 2,$ one has
$\mathbb R^k\cong V_o\oplus V_{\perp}$ and 
\begin{align}\label{eq:V_dcom}
V:=\mathbb R^{k^2}&\cong({\mathbb R^k})^{\otimes 2}=(V_o\oplus V_{\perp})\otimes (V_o\oplus V_{\perp})\notag\\
&=(V_o\otimes V_o)\oplus (V_o\otimes V_{\perp})
\oplus (V_{\perp}\otimes V_o)\oplus (V_{\perp}\otimes V_{\perp})\notag\\
&=(V_o\otimes V_o)\oplus (V_o\otimes V_{\perp})
\oplus (V_{\perp}\otimes V_o)\oplus \Big(\operatorname{Sym}^2(V_{\perp}) \;\oplus\; \wedge^2(V_{\perp})\Big)\notag,
\end{align}
where 
\begin{equation}\label{eq:2}
\operatorname{Sym}^2(V_{\perp})=\{U\in V_{\perp}^{\otimes 2}:U=sU,\; \text{for } s\in S_2\} \qquad
  \wedge^2(V_{\perp})=\{U\in V_{\perp}^{\otimes 2}:U=-sU,\; \text{for } s\in S_2\},  
\end{equation}
and is called \textit{second symmetric power} and \textit{exterior power}, respectively. (see \cite{reprbook} Section 2.11 for definition and \cite{fulton2013representation} Chapter 4.1 for derivation of $V_{\perp}\otimes V_{\perp}$.) 

\vskip 0.3cm
We next borrow the concepts from Young diagrams and Frobenius's Character Formula \eqref{eq:frobenius} to derive the general form of $S_k$ isotypic decomposition of $V:=\mathbb R^{k^2}.$ Let $\eta=(\eta_1,\eta_2,\cdots, \eta_r)$ be a partition of $k$, represented by a Young diagram whose rows have lengths \[
\eta_1\ge \eta_2\ge\cdots\ge \eta_r\ge 0,\; \sum_i^r \eta_i=k.
\]
Each such partition labels a unique irreducible representation $S^{\eta}$ of $S_k,$ and is known as \textit{Specht Module.} It is well-known fact that $V_{\perp}\cong S^{(k-1,1)}$ and from Frobenius's Character Formula \eqref{eq:frobenius}, one has for any $k\in\mathbb N^+,\; k\ge 4$ that
\begin{align}\label{eq:symwedge}
\operatorname{Sym}^2(V_{\perp})\cong S^{(k)}\oplus S^{(k-1,1)}\oplus S^{(k-2,2)},\quad
&\wedge^2(V_{\perp})\cong S^{(k-2,1,1)}.\notag 
\end{align}
Therefore, we derive the following general form for $S_k$-isotypic decomposition of $V$:
\begin{equation}\label{eq: decomposition}
V\cong2 S^{(k)}\oplus 3 S^{(k-1,1)}\oplus S^{(k-2,2)}\oplus S^{(k-2,1,1)},\qquad  k\ge 4.
\end{equation}

\noindent Notice that here $S^{(k)}, S^{(k-1,1)}$  are the trivial and standard representation, respectively. By Hook Length Formula (see \cite{reprbook} Chapter 5.17), one can also obtain
\begin{align*}
&\dim S^{(k-2,2)}=k(k-3)/2,\quad
&\dim S^{(k-2,1,1)}=(k-1)(k-2)/2.
\end{align*}
We list the detailed derivation of decomposition \eqref{eq: decomposition} in Appendix \ref{app:A2}. For more thorough understanding, one is referred to \cite{fulton2013representation} Chapter 4.1. 
\section{Hessian and 
Its Spectrum at Global Minima}\label{sec:3}
\subsection{General Form of Hessian.}\label{sec:3.1} Let $x,y\in \mathbb R^k$ be two non-parallel vectors, denote by $\theta_{xy}\in (0,\pi)$ the angle between them and $\hat{x}=\frac{x}{\|x\|},\; \hat{y}=\frac{y}{\|y\|}.$ Define
$n_{xy}=\hat{x}-\cos\theta_{xy}\hat{y},\quad \hat{n}_{xy}=\frac{n_{xy}}{\|n_{xy}\|}.$
Note, $\hat{n}_{xy}=0$ if $x,y$ are non-zero but paralleled vectors. Now let $\mathbb I_k$ be the $k\times k$ identity matrix, and define 
\begin{align*}
h_1(x,y)&:=\frac{\sin\theta_{xy}\|y\|}{2\pi\|x\|}\Big(\mathbb I_k-\frac{xx^T}{\|x\|^2}+\hat{n}_{yx}\hat{n}_{yx}^T\Big),\quad
h_2(x,y):=\frac{1}{2\pi}\Big(-\theta_{xy}\mathbb I_k+\frac{\hat{n}_{xy}y^T}{\|y\|}+\frac{\hat{n}_{yx}x^T}{\|x\|}\Big).
\end{align*}
Then one has the Hessian $\mathcal A_{\alpha}:\Omega\to \mathbb R^{k\times k}$ given by
\[
\mathcal A_{\alpha}(u):=\nabla^2_{u}\mathcal F_{\alpha}(u)=
\begin{bmatrix}
    A_{11}(u)&\cdots&A_{1k}(u)\\
    \vdots&&\vdots\\
    A_{k1}(u)&\cdots&A_{kk}(u)
\end{bmatrix},
\]
where each $A_{ij}(u), \; i,j\in \{1,\cdots,k\}$ is a $k\times k$ block matrix and $A_{ij}(u)=A^T_{ji}(u).$ In particular,
\begin{equation}\label{eq:hessian}
A_{ii}(u)=\frac{1}{2}\mathbb I_k+\sum_{j=1}^k \alpha\Big(h_1(u_i,u_j)-h_1(u_i,v_j)\Big),\quad   
A_{ij}(u)=\frac{1}{2}\mathbb I_k+\alpha h_2(u_i,u_j),\;\; i\neq j
\end{equation}
For the detailed derivation of the Hessian, see \cite{Hessian2020} Appendix C and \ref{app:hessian} for supplement.\vskip0.3cm
\noindent Moreover, at the global minimum $v^o,$ one has $\mathcal A_{\alpha}(v^o):\Omega \to  \Omega,$ and 
\begin{equation}\label{eq:spe_I}
A_{ii}( v^o)=\frac{1}{2}\mathbb{I}_k, \;\;i\in\{1,\cdots,k\},\quad
A_{ij}( v^o)=\frac{1}{2}\mathbb I_k-\frac{\alpha}{4}\mathbb I_k+\frac{\alpha}{2\pi}(E_{ij}+E_{ji}),\;\; i\neq j,
\end{equation}
where $E_{ij}=e_i{e_j}^T,\; E_{ji}=e_j{e_i}^T.$\vskip 0.3cm
\noindent \subsection{ Spectrum of Hessian at Global Minima.}\label{sec:3.2} Let $U=[u_1,u_2,\cdots, u_k]$ be $k\times k$ matrix where $u_j\in \mathbb R^k,\; j=1,\cdots,k.$ Define block matrix operator $\mathcal L_{\alpha}:\mathbb R^{k\times k}\to \mathbb R^{k\times k}$ given by
\[
(\mathcal L_{\alpha}U)_i:=\sum_{j=1}^kA_{ij}u_j,\quad i=1,\cdots,k.
\]
Notice, the block operator $\mathcal L_{\alpha}$ is the Hessian $\mathcal A_{\alpha}(v^o)$ written in block form. 
In particular, eigenvectors of $\mathcal A_{\alpha}$ correspond to block eigenmatrices of $\mathcal L_{\alpha}$. Therefore, one has 
\begin{equation}\label{eq:specturm}
    \mathcal L_{\alpha}(U)=U\Big(aJ+b\mathbb I_k\Big)+c\Big(U^T+\operatorname{tr}(U)\mathbb I_k-2\operatorname{Diag}(U)\Big),
\end{equation}
where $a=\frac{1}{2}-\frac{\alpha}{4},\; b=\frac{\alpha}{4},\;c=\frac{\alpha}{2\pi},$ and $J=\mathbf{1}\mathbf{1}^T,$  where $\mathbf{1}=(1,\cdots,1)^T\in\mathbb R^k.$ $\operatorname{tr}(U)$ denotes the trace of matrix $U.$ Given element 
$u_{ij}$ of matrix $U$, define the operator $\operatorname{Diag}:\mathbb R^{k\times k} \to \mathbb R^{k\times k}$ by 
 $\operatorname{Diag}(U)_{ij}=\begin{cases}
    u_{ii},\; i=j\\
    0,\quad i\neq j.
\end{cases}$ One is referred to Appendix \eqref{ap:spectrum} for more details of the derivation. \vskip0.3cm

Recall that equation \eqref{eq:2} can be equivalently expressed as
\[
\operatorname{Sym^2}(V_{\perp})=\{U\in V_{\perp}^{\otimes 2}:U^T=U\},\;\;\wedge^2(V_{\perp})=\{U\in V_{\perp}^{\otimes 2}:U^T=-U\}.
\]
Consider the $S_k$–equivariant map 
$\operatorname{diag}:\operatorname{Sym^2}(V_{\perp})\to \mathbb R^k\cong V_o\oplus V_\perp,$
which sends symmetric matrix to its diagonal vector. Define
$P_{\perp}:=\mathbb I_k-\frac{1}{k}J.$
Notice that the operator $P_{\perp}$ is an orthogonal projection of $\mathbb R^k$ onto $V_{\perp}$ and $P_{\perp}\in \operatorname{Sym^2(V_{\perp})}.$ Moreover, it is invariant under the diagonal action of $S_k$ and satisfies $\operatorname{diag}(P_{\perp})=(1-\frac{1}{k})\mathbf 1\in V_o.$
For any \(w\in V_\perp\), let 
\[
U_w=(w\mathbf 1^\top+\mathbf 1 w^\top)-k\,\mathrm{Diag}(w) \in\operatorname{Sym^2(V_{\perp})},
\]
thus
\(\operatorname{diag}(U_w)=(2-k)w\in V_{\perp}\). This construction shows that 
$\operatorname{diag}$ is surjective onto $V_{o}\oplus V_{\perp},$ with kernel 
$\mathcal S_0:=\ker(\operatorname{diag})
=\{U\in \operatorname{Sym^2(V_{\perp})}:\ \operatorname{diag}(U)=0\}.$
Hence, as $S_k$–representations,
\begin{align*}
\operatorname{Sym^2(V_{\perp})}
&\cong\mathcal S_0\oplus S^{(k)}\oplus S^{(k-1,1)},
\end{align*}
where $\mathcal S_0\cong\ S^{(k-2,2)}.$ Notice that 
the trivial representation 
$S^{(k)}$ corresponds to the subspace $\operatorname{span}\{P_{\perp}\}\subset \operatorname{Sym^2}(V_{\perp})$ and 
\begin{equation}\label{eq:Rkk-decomp}
V:=\mathbb R^{k^2}\ \cong\operatorname{span}\{\mathbb I_k,J\}
\ \oplus\Big(\ V_o\otimes V_{\perp}\Big)\oplus \Big(V_{\perp}\otimes V_o\Big)\oplus S^{(k-1,1)}\oplus \mathcal S_0 \oplus \wedge^2(V_{\perp}) 
.
\end{equation}
We next discuss the 3 copies of standard representation $\Big(\ V_o\otimes V_{\perp}\Big)\oplus \Big(V_{\perp}\otimes V_o\Big)\oplus S^{(k-1,1)}.$
Fix a basis \(r=(r_1,\dots,r_{k-1})^T\) of \(V_\perp\) (e.g.\ \(r_i=e_i-e_k\)).
For each \(i\), put
\begin{equation}\label{eq:KSD}
K_i:=r_i\mathbf 1^\top-\mathbf 1 r_i^\top,\qquad
S_i:=r_i\mathbf 1^\top+\mathbf 1 r_i^\top,\qquad
D_i:=\mathrm{Diag}(r_i)-\frac{1}{k}S_i,\qquad
\mathcal W_i:=\operatorname{span}\{K_i,S_i,D_i\}.
\end{equation}
Then we have the following properties
\begin{align}\label{eq:standard}
K_iJ=S_iJ=kr_i\mathbf 1^T,\; 2r_i\mathbf 1^T=S_i+K_i,\; 2\mathbf 1r_i^T=S_i-K_i,\; D_iJ=0.
\end{align}
Moreover, it's obvious that $K_i, S_i$ span the two copies of $S^{(k-1,1)}$ in $V_o\otimes V_{\perp}\oplus V_{\perp}\otimes V_o,$ anti-symmetric and symmetric parts, respectively, and $D_i$ span the copy in $\operatorname{Sym^2}(V_{\perp})$. Therefore,
\[
\bigoplus_{i=1}^{k-1}\mathcal W_i \cong \Big(\ V_o\otimes V_{\perp}\Big)\oplus \Big(V_{\perp}\otimes V_o\Big)\oplus S^{(k-1,1)}
\]
and
\begin{equation}\label{eq:Rkk-decomp_com}
V:=\mathbb R^{k^2}\cong\operatorname{span}\{\mathbb I_k,J\}
\ \oplus \bigoplus_{i=1}^{k-1}\mathcal W_i\oplus\mathcal S_0\oplus\ 
\wedge^2(V_{\perp}).
\end{equation}
We next compute the eigenvalues of $\mathcal L_{\alpha}$ based on the matrix identification in equation \eqref{eq:Rkk-decomp_com}. By direct computation using equation \eqref{eq:specturm}, one has
\begin{enumerate}
    \item[a)] For $U\in \wedge^2(V_{\perp}): \mathcal L_{\alpha}(U)=(b-c)U=(\frac{\alpha}{4}-\frac{\alpha}{2\pi})U,$ with multiplicity $\frac{(k-1)(k-2)}{2}.$ 
    \item[b)] For $U\in \mathcal S_0: \mathcal L_{\alpha}(U)=(b+c)U=(\frac{\alpha}{4}+\frac{\alpha}{2\pi})U,$ with multiplicity $\frac{(k-3)k}{2}.$   
    \item[c)] $U\in \operatorname{span}\{\mathbb I_k,J\}:$ $\mathcal L_{\alpha}(\mathbb I_k)=aJ+(b+ck-c)\mathbb I_k$ and $\mathcal L_{\alpha}(J)=c(k-2)\mathbb I_k+(ka+b+c)J,$ which implies the restriction $\mathcal L_{\alpha}|_{\operatorname{span}\{\mathbb I_k,J\}}$ is given by
    \[
    \begin{bmatrix}
b+c(k-1) & a\\[2pt]
c(k-2) & ak+b+c
\end{bmatrix}
\]
with eigenvalues
$\frac12\!\left(2b+k(a+c)\ \pm\ \sqrt{\,k^2(a-c)^2+4c(2a-c)(k-1)\,}\right)$ and each multiplicity $1.$
\item[d)] For $U\in \operatorname{span}{\mathcal W_i:}$ By applying properties \eqref{eq:standard}, one has
\begin{equation}\label{eq:blockW}
\begin{aligned}
\mathcal L_{\alpha}(K_i)&=\Big(b-c+\tfrac{ak}{2}\Big)K_i+\tfrac{ak}{2}\,S_i,\quad
\mathcal L_{\alpha}(S_i)=\tfrac{ak}{2}\,K_i+\Big(b+c+\tfrac{ak}{2}-\tfrac{4c}{k}\Big)S_i-4c\,D_i,\\
\mathcal L_{\alpha}(D_i)&=-\tfrac{2c}{k^2}\,S_i+(b+c-\tfrac{2c}{k})\,D_i.
\end{aligned}
\end{equation}
Hence, in the ordered basis \(\{K_i,S_i,D_i\}\), the restriction \(\mathcal L|_{\mathcal W_i}\) is given by
\[
\begin{bmatrix}
b-c+\frac{ak}{2} & \frac{ak}{2} & 0\\[3pt]
\frac{ak}{2} & b+c+\frac{ak}{2}-\frac{4c}{k} & -4c\\[3pt]
0 & -\frac{2c(k-2)}{k^2} & b-c+\frac{4c}{k}
\end{bmatrix},
\]

whose eigenvalues are
\[\ b-c,\qquad
\frac12\!\left(ak+2b\ \pm\ \sqrt{a^2k^2+4c(c-2a)}\right)\ 
\]
\end{enumerate}
Therefore, for $k\ge 4$, we obtain the spectrum $\lambda_{S^{(\eta)}}(\alpha)$ of loss $\mathcal L _{\alpha}$ listed in table \eqref{tab:spectrum}, where $S^{(\eta)}$ denotes the Specht module with partition $\eta$: 
\begin{table}[h!]
\centering
\renewcommand{\arraystretch}{0.75}
\begin{tabular}{|c|c|}
\hline
$\lambda_{S^{(\eta)}}(\alpha)\in \operatorname{Spec(\mathcal L_{\alpha}})$ & Multiplicity \\
\hline\hline
$b-c$ & $k(k-1)/2$ \\
$b+c$ & $k(k-3)/2$ \\
$\displaystyle 
\frac{1}{2}\left(ak+2b\ \pm\ \sqrt{a^{2}k^{2}+4c(c-2a)}\right)$
& $k-1$ each \\
$\displaystyle
\frac{1}{2}\left(2b+k(a+c)\ \pm\ \sqrt{k^{2}(a-c)^{2}+4c(2a-c)(k-1)}\right)$
& $1$ each \\
\hline
\end{tabular}
\caption{Spectrum of $\mathcal L$}
\label{tab:spectrum}
\end{table}
\section{Critical Sets and Bifurcation Result}\label{sec:4}
\subsection{Critical Set.}\label{sec:4.1}
Recall that our main objective is to identify non-constant solutions to system\eqref{eq: system} emerging from global minima $v^o$. The forthcoming analysis is based on the Slice Criticality Principle (see Theorem~\eqref{thm:SCP}), which provides a framework for computing the equivariant gradient degree of $\nabla_u \mathcal F_{\alpha}$. \vskip0.3cm
We define the restriction
$\mathscr F(\alpha, u):=\mathcal F_{\alpha}(u)|_{S_o},\;\; \alpha\in\mathbb R.$
By construction, $\mathscr F$ is invariant under the isotropy $G_{v^o}$. This restriction allows us to apply the Slice Criticality Principle in a small neighborhood $\mathcal U$ of $G(v^o)$ in order to compute the equivariant gradient degree of $\nabla_u \mathcal F_{\alpha}$.\vskip0.3cm
Next, we introduce the linearization operator
\begin{equation}\label{eqn:linearization}
\mathscr L(\alpha):=\nabla_{u}^{2}\mathscr F(\alpha,v^{o}):S_o\to S_o.
\end{equation}
Since $
\mathscr L(\alpha)=
\nabla_{u}^{2}\mathscr F(\alpha,v^{o}),$ one verifies that $G(v^{o})$ forms a finite-dimensional isolated orbit of critical points of $\mathscr F$ whenever $\mathscr L(\alpha)$ is an isomorphism. Consequently, if a pair $(\alpha^{o},v^{o})$ represents a bifurcation point of system~\eqref{eq: system}, then $\mathscr L(\alpha^{o})$ must fail to be an isomorphism. We therefore define the critical set associated with $v^o$ as
$\Lambda:=\{\alpha \in \mathbb R:\mathscr L(\alpha)\text{ is not an isomorphism}\}.$
From the spectrum presented in Table \eqref{tab:spectrum}, it follows that the critical set in our case is given by $\Lambda=\Big\{0,\alpha_1(k), \alpha_2(k)\Big\},$ $k\in\mathbb N^+,\; k\geq 4,$ where
\[
\alpha_1(k)=\frac{8\pi+2k\pi^2}{4+(k-1)\pi^2+4\pi},\qquad \alpha_2(k)=\frac{2\pi(4-4k+2k^2+k\pi)}{(k-1)(2k\pi+\pi^2-4\pi-4)}.
\]
\begin{figure}[H]
  \centering
\includegraphics[width=0.6\textwidth, keepaspectratio]{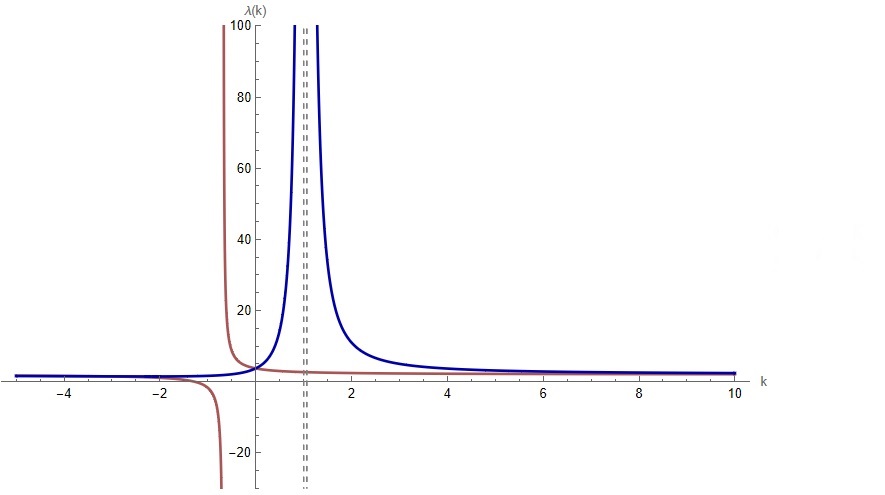}
  \caption{critical numbers $\alpha_{j}(k),\; j=1,2$}
  \label{fig:criticalnum}
\end{figure}
\noindent A direct analysis shows that for $k \ge 4$, the critical number 
$\alpha_{S^{(k-1,1)}}$ remains strictly less than 
$\alpha_{S^{(k)}}$, with both approaching the horizontal 
asymptote $\alpha = 2$ as $k \to \infty$. Therefore, by combining results from section \eqref{sec:3.2}, one has the following relation among different critical numbers:
\[
\alpha_{S^{(k-2,2)}}=\alpha_{S^{(k-2,1,1)}}=\alpha_{S^{(k-1,1)}}^1<\alpha_{S^{(k-1,1)}}^2=\alpha_{S^{(k-1,1)}}^3<\alpha_{S^{(k)}}^1=\alpha_{S^{(k)}}^2,
\]
Notice, these critical numbers are not uniquely identified by the partitions due to resonance. 
\subsection{Equivariant Bifurcation for Any Number of Neurons \texorpdfstring{$k\ge 4$}{k≥4}} Note that the bifurcation invariant $\omega_G(\alpha_{S^{(\eta)}})$ takes values in the Euler ring $U(S_k).$ This formulation enables a full characterization of symmetry types associated with Weyl groups of nonzero dimension. In the present context, however, since $S_k$ is discrete group so the computation can be carried out by restricting to the Burnside ring $A(S_k)$. The main equivariant bifurcation result for fully connected networks is stated as follows. 
\begin{theorem}\label{thm:main}
Consider fully connected two-layer teacher–student neural networks with Gaussian inputs and leaky ReLU, both input and hidden width $k\ge 4,$ 
\begin{enumerate}
\item[i)] The system \eqref{eq: system}
admits branches of critical points emerging from the global minima $v^o,$ with symmetry types corresponding to the following four $S_k$ Specht isotypic components: $S^{(k)},S^{(k-1,1)},S^{(k-2,2)},S^{(k-2,1,1)}$. Typically, there are three distinct branches bifurcating from $v^o$ when $\alpha$ cross zero, each exhibiting the symmetry from one of the representations $S^{(k-1,1)}$, $S^{(k-2,2)}$, and $S^{(k-2,1,1)}$.
\item[ii)] The bifurcation only occurs when leaky slope $\alpha$ is nonnegative and the bifurcation threshold is width-invariant.
\item[iii)]  The global minimum undergoes no further symmetry-breaking instability
throughout the engineering regime $\alpha \in (0, 1).$
\end{enumerate}
\end{theorem}
\begin{proof}
\begin{enumerate}
\item[i)] Define $\mathscr K_c:=\{\eta:\alpha_{S(\eta)}=\alpha_c\}.$ For each $c_o\in \Lambda$ such that the interval $\alpha_-<\alpha_{c_o}<\alpha_+$ contains no other critical numbers, i.e. $[\alpha_-, \alpha_+]\cap\Lambda=\{\alpha_{c_o}\}.$ As established in Section \ref{sec:4.1}, there exists an isolated tubular neighborhood $\mathcal U$ of $G(v^o)$ such that $\Bar {\mathcal U}$ contains no other critical orbits of $\mathcal F_{\alpha_{\pm}}$. Applying the Slice Principle (Theorem \ref{thm:SCP}), one obtains
$\nabla_{G}\text{\textrm{-deg}}\Big(\nabla\mathcal F_{\alpha_{\pm}}
,\mathcal U\Big)=\Theta\Big(\nabla_{G_{v^o}}\text{\textrm{-deg}}(\nabla \mathcal F_{\alpha_{\pm}}
,\mathcal U\cap S_o)\Big),$
where in the present setting $G=S_k\times S_k,\; G_{v^o}=\triangle S_k,$ and $\Theta: U(G_{v^o})\to U(G)$ is a homomorphism defined by $\Theta(H)=(H)$ for each orbit type $(H)\in \Phi_0(G).$ Hence the associated topological invariant $\omega_G(\alpha_{c_o})$ takes the form
\begin{equation}\label{eqn:invariant}
\omega_G(\alpha_{c_o})=\nabla_{G_{v^o}}\text{\textrm{-deg}}(\nabla\mathcal F_{\alpha_{-}}
,\mathcal U\cap S_o)-\nabla_{G_{v^o}}\text{\textrm{-deg}}(\nabla\mathcal F_{\alpha_{+}}
,\mathcal U\cap S_o),
\end{equation}
where, for brevity, $\Theta$ is omitted in the notation. \vskip0.3cm
\noindent If this invariant expands as
$\omega_G(\alpha_{c_o})=m_1(H_1)+\cdots m_r(H_r),$
with nonzero coefficients $m_i\neq 0, \; (i=1,\cdots, r),$ then it follows that branches of nontrivial solutions bifurcate from $v^o$ with symmetry at least $(H_i).$ Our objective is therefore to determine, for each $\alpha_{c_o}\in \Lambda$, the general expression of $\omega_G(\alpha_{c_o})$ and, in particular, the coefficients $m_i$ for each $(H_i).$ \vskip0.3cm

By linearization (see \eqref{eqn:linearization}) and its computation based on $G$-equivariant basic degree \eqref{eq:grad-lin}, one can derive the following
$
\nabla_{G_{v^o}}\text{\textrm{-deg}}(\nabla\mathcal F_{\alpha_{\pm}}
,\mathcal U\cap S_o)=\nabla_{G_{v^o}}\text{\textrm{-deg}}(\mathcal A_{\alpha_{\pm}}
,\mathcal U\cap S_o)
$
and 
\[
\nabla_{G_{v^o}}\text{\textrm{-deg}}(\mathcal A_{\alpha_{-}}
,\mathcal U\cap S_o)=\prod_{\{\eta:\;\alpha_{S^{(\eta)}}<\alpha_{c_o}\}}\nabla\text{\textrm{-deg}}_{\mathcal W_{\eta}}^{m_{\eta}(\alpha_{S^{(\eta)}})},
\]
\[
\nabla_{G_{v^o}}\text{\textrm{-deg}}(\mathcal A_{\alpha_{+}}
,\mathcal U\cap S_o)=\prod_{\{\eta:\;\eta\in \mathscr K_{c_o}\}}\nabla\text{\textrm{-deg}}_{\mathcal W_{\eta}}^{m_{\eta}(\alpha_{S^{(\eta)}})}\prod_{\{\eta:\;\alpha_{S^{(\eta)}}<\alpha_{c_o}\}}\nabla\text{\textrm{-deg}}_{\mathcal W_{\eta}}^{m_{\eta}(\alpha_{S^{(\eta)}})},
\]
Notice that each $\nabla\text{\textrm{-deg}}_{\mathcal W_{\eta}}$ represents the basic degree, which can be directly computed using \textsc{G.A.P} for a given $k$. From Sections~\eqref{sec:2.2} and~\eqref{sec:3.2}, we also have that
$m_{\eta}(\alpha_{S^{(\eta)}})=1$. Consequently, the following expression for $\omega_G(\alpha_{c_o})$ can be obtained:
\begin{equation}\label{alg:algorithm_invariant}
\omega_G(\alpha_{c_o})=\prod_{\{\eta:\;\alpha_{S^{(\eta)}}<\alpha_{c_o}\}}\nabla\text{\textrm{-deg}}_{\mathcal W_{\eta}}\Big((S_k)-\prod_{\{\eta:\eta\in\mathscr K_{c_o}\}}\nabla\text{\textrm{-deg}}_{\mathcal W_{\eta}}\Big).
\end{equation}
Moreover, note that the critical number~$0$ corresponds to three distinct $S_k$-isotypic components $S^{(k-1,1)}$, $S^{(k-2,2)}$ and $S^{(k-2,1,1)}$, each giving rise to a separate branch of bifurcating critical points.\vs
\noindent Results ii) and iii) can be easily derived from direct analysis in Section \ref{sec:4.1} and Figure \ref{fig:criticalnum}.
\end{enumerate}
\end{proof}
\section{Numerical Example}\label{sec:5} In this section, we illustrate how the framework in Theorem \eqref{thm:main} can be applied to identify the symmetries of bifurcating critical points. We take the case $k=5$ here for demonstration. However, we need to emphasize that the same procedure applies to any $k\ge 4$. \vskip0.3cm
\noindent Notice that in our case, the space $V:=\mathbb R^{25}$ is a representation of the group $G=S_5\times S_5,$ and the action of $G$ on $\mathbb R^{25}$ is given by
$(\sigma, \gamma)
(v_1,v_2,\cdots,v_5)^T=(\gamma v_{\sigma(1)}, \gamma v_{\sigma(2)}, \cdots, \gamma v_{\sigma(5)})^T.$

Consider isotropy group of global minima
$\triangle{S}_5:=\left\{(\sigma,\sigma): \sigma \in S_5 \right\}.$
 The purpose of the following is to obtain the $\triangle{S}_5$ isotypic decomposition of the slice $S_o$ which is equivalent to $V=\mathbb R^{25}$. We first study the character table of ${S}_5,$ as shown in Table \eqref{tbl:characters_s5}, where $\chi_j,j=1,\cdots,7$ denotes the characters of all $S_5$ irreducible representation $\mathcal{W}_j$ and $\mathcal{\chi}_{V}$ is the character of $V.$

\begin{table}[H] 
\centering
\resizebox{0.8\textwidth}{!}{
\begin{tabular}{|c|c|c|c|c|c|c|c|c|}
\hline
Rep. & Character & (1)& (12) & (12)(34) &(123) & (123)(45) &(1234) & (12345) \\\hline
$\mathcal{W}_{1}$ & $\chi_{1}$ & 1 &-1  &1  &1 &-1 &-1  &1\\
$\mathcal{W}_{2}$ & $\chi_{2}$ & 4 &-2  &0  &1  &1  &0 &-1\\
$\mathcal{W}_{3}$ & $\chi_{3}$ & 5 &-1  &1 &-1 &-1 & 1 & 0\\
$\mathcal{W}_{4}$ & $\chi_{4}$ & 6  &0 &-2  &0  &0  &0  &1\\
$\mathcal{W}_{5}$ & $\chi_{5}$ & 5  &1 & 1 &-1 & 1 &-1 & 0\\
$\mathcal{W}_{6}$ & $\chi_{6}$ & 4  &2  &0  &1 &-1 & 0 &-1\\
$\mathcal{W}_{7}$ & $\chi_{7}$ & 1 & 1 & 1 & 1 & 1 & 1 & 1\\\hline
$V=\mathbb{R}^{25}$ & $\chi_{V}$ & $25$ &9  &1  &4  &0 & 1& 0
\\\hline
\end{tabular}
}
\caption{Character Table of $S_5$}
\label{tbl:characters_s5}
\end{table}

\noindent By direct computation,
$V=\mathbb R^{25}=\mathcal W_{4}\oplus \mathcal W_{5}\oplus 3\mathcal W_{6}\oplus 2\mathcal W_{7}.$
Notice $\mathcal W_{7}$ is the trivial representation, $\mathcal W_{6}$ can be identified with standard representation given that $\dim \mathcal W_{6}=4.$ Similarly, $\mathcal W_{5}\cong S^{(3,2)}$ and $\mathcal W_{4}\cong S^{(3,1,1)}.$ By applying formula \eqref{eq:spe_I}, we conclude that the spectrum of the Hessian consists of
\begin{equation}
\label{eq:spectrum}\sigma(\nabla^{2}_u\mathcal F_{\alpha}(v^{o}))=\begin{cases}
\frac{\alpha}{4}-\frac{\alpha}{2\pi}  \qquad  \text{
with mult= }10\\
\frac{\alpha}{4}+\frac{\alpha}{2\pi}\qquad  \text{ with mult }5\\
\frac{\pi(10-3\alpha)\pm \rho_1}{8\pi}\quad \text{ with mult=} 4\\
\frac{\pi(10-3\alpha)+10\alpha\pm \rho_2}{8\pi}\quad \text{ with mult=} 1\\
\end{cases}
\end{equation}
where $\rho_1=\sqrt{25\pi^2(\alpha-2)^2+16\pi\alpha(\alpha-2)+16\alpha^2},\quad \rho_2=\sqrt{25\pi^2(\alpha-2)^2+36\pi\alpha(\alpha-2)+36\alpha^2}.$
Moreover, we have
$\mathcal A_{\alpha}|_{\mathcal W_{j}}=\lambda_j  \text{\textrm{Id\,}}$
Thus $\mathcal A_{\alpha}|_{\mathcal W_{j}}=0$ if and only if $\lambda_j(\alpha)=0$ for $j=4,5,6,7$. Notice that the correspondence between index $j$ and partition $\eta$ are as discussed above. We denote the critical numbers $\alpha\in\Lambda$ as
$\alpha_{j}\in\ker(\lambda_j)~,$
and the critical set $\alpha$ associated with the equilibrium $v^{o}$ of the system
\eqref{eq: system} is described as
$\Lambda:=\left\{ \alpha_j\in \ker(\lambda_j):j=4,5,6,7\right\} .$
\noi For our specific case, one can derive the relation among different $\alpha_{j}:$
\begin{align}\label{sq:ordering}
0=\alpha_{4}=\alpha_{5}=\alpha_{6}^1<\alpha_{6}^2=\alpha_{6}^3<\alpha_{7}^1=\alpha_{7}^2\approx 3.1587.
\end{align}
\textbf{Computation of the Gradient Degree.}
The following are the basic degrees in $A(S_5)$ computed by G.A.P (see \cite{balanov2025degree}), with maximal orbit types in each isotypic component noted in red.
\begin{align}\label{eqn:basicdegree}
\nabla\text{-deg}_{\mathcal W_{4}}=&-(\mathbb Z_1)+2(D_1)+(\mathbb Z_2)+(\mathbb Z_3)-\textcolor{red}{(\mathbb Z_4)}-\textcolor{red}{(D_2)}-\textcolor{red}{(D_3)}-\textcolor{red}{(\mathbb Z_6)}+(S_5),\\
\nabla\text{-deg}_{\mathcal W_{5}}=&-(D_2)+(V_4)+3(D_2)-\textcolor{red}{2(D_4)}-\textcolor{red}{(D_5)}-\textcolor{red}{2(D_6)}+(S_5),\notag\\
\nabla\text{-deg}_{\mathcal W_{6}}=& (\mathbb Z_1)-4(D_1)+3(D_2)+3(D_3)-\textcolor{red}{2(D_6)}-\textcolor{red}{2(S_4)}+(S_5),\notag\\
\nabla\text{-deg}_{\mathcal W_{7}}=&-\textcolor{red}{(S_5)}.\nonumber
\end{align}
Then by applying formula \eqref{alg:algorithm_invariant}, one can derive the following bifurcation invariants:
\begin{align}\label{eq:invariant}
\omega_{G}(\alpha_{4}) =\omega_{G}(\alpha_{5})=\omega_{G}(\alpha_{6}^1)  =&(S_5)-\nabla\text{\textrm{-deg}}_{\mathcal{W}_{4}%
}*\nabla\text{\textrm{-deg}}_{\mathcal{W}_{5}%
}*\nabla\text{\textrm{-deg}}_{\mathcal{W}_{6}%
},\\
\omega_{G}(\alpha_{6}^{2})=\omega_{G}(\alpha_{6}^{3}) =&\nabla\text{\textrm{-deg}}_{\mathcal{W}_{4}%
}*\nabla\text{\textrm{-deg}}_{\mathcal{W}_{5}%
}*\nabla\text{\textrm{-deg}}_{\mathcal{W}_{6}%
}\ast\Big((S_{5})-\nabla\text{\textrm{-deg}}_{\mathcal{W}_{6}}\Big)\nonumber\\=&\nabla\text{\textrm{-deg}}_{\mathcal{W}_{4}%
}*\nabla\text{\textrm{-deg}}_{\mathcal{W}_{5}%
}*\nabla\text{\textrm{-deg}}_{\mathcal{W}_{6}%
}-\nabla\text{\textrm{-deg}}_{\mathcal{W}_{4}%
}\ast\nabla\text{\textrm{-deg}}_{\mathcal{W}_{5}},\nonumber\\
\omega_{G}(\alpha_{7}^{1})=\omega_{G}(\alpha_{7}^{2}) =&\nabla\text{\textrm{-deg}}_{\mathcal{W}_{4}%
}\ast\nabla\text{\textrm{-deg}}_{\mathcal{W}_{5}}\ast\Big((S_5)-\nabla
\text{\textrm{-deg}}_{\mathcal{W}_{7}}\Big)\nonumber\\
=&\nabla\text{\textrm{-deg}}_{\mathcal{W}_{4}%
}\ast\nabla\text{\textrm{-deg}}_{\mathcal{W}_{5}}-\nabla\text{\textrm{-deg}}_{\mathcal{W}_{4}%
}\ast\nabla\text{\textrm{-deg}}_{\mathcal{W}_{5}}\ast\nabla
\text{\textrm{-deg}}_{\mathcal{W}_{7}}.\nonumber
\end{align}
Using GAP and the algorithm described in section \ref{sec:E}, we compute the bifurcation invariant $\omega_{G}(\alpha_{j_o})$ reduced to $A(S_5).$ Note that all maximal isotropy types are highlighted in red:%
\begin{align}
\omega_{G}(\alpha_{4})=\omega_{G}(\alpha_{5})=\omega_{G}(\alpha_{6}^1)&=(D_1)-2(\mathbb Z_2)+(V_4)+(\mathbb Z_4)+(D_2)-2(D_3)+(\mathbb Z_6)\label{bifurcaitoninvariant}\\
&-2(D_4)+\textcolor{red}{(D_5)}+\textcolor{red}{2(S_4)},\notag\\
\omega_{G}(\alpha_{6}^{2})=\omega_{G}(\alpha_{6}^{3})&= 2(\mathbb Z_2)+(\mathbb Z_3)-2(V_4)-2(\mathbb Z_4)-3(D_2)+(D_3)-2(\mathbb Z_6)\\
&+4(D_4)+\textcolor{red}{2(D_6)}
-\textcolor{red}{2(S_4)},\notag\\
\omega_{G}(\alpha_{7}^{1})=\omega_{G}(\alpha_{7}^{2})&=-2(D_1)-2(\mathbb Z_3)+2(V_4)+2(\mathbb Z_4)+4(D_2)+2(D_3)+2(\mathbb Z_6)\\
&-4(D_4)-2(D_5)-4(D_6)+\textcolor{red}{2(S_5)}.\notag
\end{align}
\textbf{Interpretation of the Bifurcation Invariant Obtained Above.} To briefly illustrate the meaning of the bifurcation invariants obtained in this section, we consider the maximal orbit type \((S_4)\) detected in the local invariant $\omega_G(\alpha_4)$ in equation \eqref{bifurcaitoninvariant} as an example. A direct computation shows that the elements of the $S_4$ fixed point subspace of $V$ take the form
\begin{align}\label{eq:fixedsubspace}
u_1&=(a,b,b,b,e),\; u_2=(b,a,b,b,e),\\
u_3&=(b,b,a,b,e),\; u_4=(b,b,b,a,e),\notag\\
u_5&=(f,f,f,f,g),\notag
\end{align}
where $a,b,c,d,e$ are pairwisely distinct.
Then \textcolor{red}{\(+2(S_4)\)} in \(\omega_G(\alpha_4)\) indicates that, as \(\alpha\) crosses critical number \(\alpha_4\), at least \textcolor{red}{2} distinct \(G\)-orbits of solutions of form \eqref{eq:fixedsubspace} emerge from the trivial critical orbit \(G(v^\circ)\), each orbit consists of \textcolor{red}{$5$} distinct solutions but all their loss functions take the same value.
\section{Future Work}
This work represents the first attempt to apply the methodology of equivariant degree to the study of neural-network loss landscapes and their associated training dynamics. We note that, although the present analysis is carried out under a teacher–student architecture with a perfectly symmetric teacher model and a Gaussian input distribution, these assumptions are not intrinsic to the equivariant-degree framework. In particular, the setting can be relaxed to orthogonally invariant input distributions and linearly independent teacher neurons, in which case equivariant degree theory remains applicable, although the underlying symmetry is described by a more complicated group, such as \(S_k\times O(d)\), where \(k\) denotes the number of hidden neurons and \(d\) the input dimension. Beyond neuron permutation symmetries in fully connected multilayer perceptrons, other architectures possess even richer symmetry structures. Prominent examples include heads permutation and self attention arising in Transformer for more details). Extending the equivariant-degree framework to these broader symmetry classes and architectures, and understanding their implications for loss landscapes and training dynamics, remain promising directions for future investigation.
\appendix
\section{Explicit Form of Loss Function.}\label{app:A}
\subsection{Some Preliminaries of Probability Distribution}
\begin{definition}
Let $ x\in \mathbb{R}^k$ be random vector following distribution $\mathcal D,$ then $\mathcal D$ is called orthogonally invariant if its corresponding probability density function has the property:
\[
p( x)=p(g x),\;\; g\in O(k). 
\]   
\end{definition}
\noi Notice that the standard Gaussian distribution $ \mathcal{N}(0,\mathbb{I}_k)$ is orthogonally invariant. Indeed, for $ x\sim \mathcal{N}(0,\mathbb{I}_k),$ the probability density function is given by
\begin{equation}
p(x)=\frac{1}{(2\pi)^{\frac{k}{2}}}e^{-\frac{
\|x\|^2}{2}}.
\end{equation}
\begin{definition}
Let $ x\in \mathbb{R}^k, x\sim \mathcal{D}$ and $h:\mathbb{R}^k\to \mathbb{R},$ the expectation of $h(x)$ has the form:
\[
\mathbb{E}_{ x\sim \mathcal{D}}[h( x)]=\int_{\mathcal{D}} h(x)p(x)d x,
\]
where $p( x)$ is the probability density function.
\end{definition}
\begin{lemma}\label{le:dist}
$ x\sim \mathcal{N}(0,\mathbb{I}_k)$ implies $g^T x\sim \mathcal{N}(0,\mathbb{I}_k),\; g\in O(k).$ Indeed, for $x\in \mathbb R^k,$ one has 
\[
\mathbb E[g^Tx]=g^T \mathbb E[ x]=0,
\] and 
\begin{align*}
\text{Cov}[g^T x]&=\mathbb E[(g^Tx)(g^T x)^T]=\mathbb E[(g^T x  x^Tg]\\
&=g^T \mathbb E[ x  x^T]g=g^T\text{cov}[x]g\\
&=g^T\mathbb{I}_kg\\
&=\mathbb{I}_k
\end{align*}
\end{lemma}
\subsubsection{Explicit Form of $f_{\alpha}( w,v)$}
Let $f_{\alpha}:\mathbb{R}^k\times \mathbb{R}^k \to \mathbb{R}$ given by
\begin{equation}\label{eq:expectation}
f_{\alpha}( w, v)=\mathbb{E}_{x\sim {\mathcal{N}(0,\mathbb{I}_k)}}\Big(\sigma_{\alpha}(w^Tx)\sigma_{\alpha}( v^T x)\Big),\;\; \;x\in \mathbb{R}^k,
\end{equation}
where $\sigma_{\alpha}(a)=\max\{(1-\alpha)a,a\},\; a\in \mathbb{R}$ is the leaky ReLU activation function, one has:
\begin{lemma}\label{le:property_f}(Properties of $f$)
\begin{itemize}
   \item[i)] $f_{\alpha}(w, v)$ is positively homogeneous. i.e. $f_{\alpha}(\delta  w, \gamma v)=\delta \gamma f_{\alpha}( w, v),\;\; \delta, \gamma \geq 0.$
   \item[ii)] $f_{\alpha}( w,  v)$ is $O(k)$ invariant. i.e. $f_{\alpha}(gw,g v)=f_{\alpha}( w, v),\;\;  w, v \in \mathbb{R}^k, g\in O(k).$
\end{itemize}
\end{lemma}
\begin{proof}
\begin{align*}
    f_{\alpha}(\delta w, \gamma v)
&=\mathbb{E}_{x\sim {\mathcal{N}(0,\mathbb{I}_k)}}\Big(\sigma_{\alpha}(\delta  w^Tx)\sigma_{\delta}(\gamma v^Tx)\Big)=\int_{\mathcal{N}(0,\mathbb{I}_k)}\sigma_{\alpha}(\delta w^Tx)\sigma_{\alpha}(\gamma v^T x)p(x)d x\\
&=\delta \gamma \int_{\mathcal D}\sigma(w^Tx)\sigma_{\alpha}(v^Tx)p( x)dx\\
&=\delta\gamma f_{\alpha}( w, v)
\end{align*}
On the other hand, apply Lemma \eqref{le:dist}, one has
\begin{align*}
    f_{\alpha}(g w, g  v)
&=\mathbb{E}_{x\sim {\mathcal{N}(0,\mathbb{I}_k)}}\Big(\sigma_{\alpha}((gw)^Tx)\sigma_{\alpha}((gv)^Tx)\Big)=\int_{\mathcal{N}(0,\mathbb{I}_k)}\sigma_{\alpha}(w^T(g^Tx))\sigma_{\alpha}( v^T(g^Tx))p(g^Tx)d x\\
&\overset{ y=g^Tx}{=}\int_{\mathcal{N}(0,\mathbb{I}_k)}\sigma_{\alpha}(w^T y)\sigma_{\alpha}( v^Ty)p( y)d y,\;\;\; g\in O(k)\\
&=f_{\alpha}( w, v)
\end{align*}
\end{proof}
\begin{proposition}
    For non-zero vectors $ w,v\in \mathbb{R}^k,$ $f_{\alpha}( w, v)$ in equation \eqref{eq:expectation} has the explicit form:
    \begin{equation*} 
        f_{\alpha}(w,v)=\frac{1}{2\pi}\|w\|\|v\|\Big(\alpha^2(\sin\theta-\theta\cos\theta)+(2+\alpha^2-2\alpha)\pi\cos\theta\Big), \;\; \theta=\cos^{-1}\frac
        {w\cdot v}{\|w\|\| v\|}.
    \end{equation*}
\end{proposition}
\begin{proof}
   From Lemma \eqref{le:property_f} i), one can first assume $\|\tilde{w}\|=\|\tilde{ v}\|=1.$ From Lemma \eqref{le:property_f} ii), one can assume that $\tilde{ v}={\begin{bmatrix}
       1,0,\cdots,0
   \end{bmatrix}}^T, \tilde{ w}={\begin{bmatrix}
       \cos \theta,\sin \theta,\cdots,0
   \end{bmatrix}}^T \in \mathbb{R}^k.$ Then,
   \[
   f_{\alpha}( w, v)=\| w\|\| v\|f_{\alpha}(\tilde{ w},\tilde{ v}),
   \]
and 
\begin{align*}
f_{\alpha}(\tilde{ w},\tilde{ v})
&= \int_{\mathbb{R}^2}\sigma_{\alpha}({\tilde{ w}^T x})\sigma_{\alpha}({\tilde{ v}^T x})p( x)d x\\
&\overset{ x=(x_1,x_2)}{=}\iint_{\substack{x_1\geq 0\\ x_1 \cos \theta+x_2\sin \theta\geq 0}}\Big(x_1^2\cos\theta+x_1x_2\sin\theta\Big)p(x_1,x_2)dx_1dx_2,
\end{align*}
where $p(x_1,x_2)=\frac{1}{2\pi}e^{-\frac{x_1^2+x_2^2}{2}}.$\vskip0.3cm
Let $x_1=r\cos\varphi, x_2=r\sin\varphi,$ then one has:
\begin{align*}
f_{\alpha}(\tilde{ w},\tilde{ v})
&= \iint_{\substack{\cos \varphi\geq 0\\ \cos(\varphi-\theta)\geq 0}}\Big(r^2\cos^2\varphi\cos\theta+r^2\cos\varphi\sin\varphi\sin\theta\Big)\frac{1}{2\pi}p(r)rdrd\varphi,\;\; p(r)=e^{-\frac{r^2}{2}}\\
&=\Big(\int_0^\infty r^3p(r)dr\Big)\Big(\frac{1}{2\pi}\int_{\theta-\frac{\pi}{2}}^\frac{\pi}{2} \cos\theta \cos^2\varphi+\sin \theta\cos\varphi\sin\varphi d\varphi\Big)\\
&=2\Bigg[\frac{1}{4\pi}\Big((\pi-\theta)\cos\theta+\sin\theta\Big)\Bigg]
=\frac{1}{2\pi}\Big(\sin\theta+(\pi-\theta)\cos\theta\Big).\end{align*}
Therefore,
\[
f_{\alpha}( w,v)=\frac{1}{2\pi}\|w\|\| v\|\Big(\sin\theta+(\pi-\theta)\cos\theta\Big).
\]
One can also refer to \cite{arjevani2021analytic} for the detail of the proof.
\end{proof}
\subsection{Explicit Form of Loss Function \texorpdfstring{$\mathcal F_{\alpha}(W)$}{Fα(W)}}\label{app:explicit_form}
The basic idea is to find $
W$ that minimizes the distance between $W$ and
$V$, where the distance in this work is measured using the MSE method. Suppose $x\sim \mathcal{N}(0,\mathbb I_k),$ the loss function $\mathcal{F}_{\alpha}:\mathbb R^{k^2}\to \mathbb{R}$ is given by \begin{equation*}
    \mathcal{F}_{\alpha}(W):=\mathcal{L}_{\alpha}(W,V)=\frac{1}{2}\mathbb{E}_{\mathbf x\sim \mathcal{N}(0,\mathbb I_k)}\Big(\sum_{i=1 }^{k}\sigma_{\alpha}({w_i}^T x)-\sum_{i=1}^{s}\sigma_{\alpha}({v_i}^T{x}) \Big)^2,
\end{equation*}

\begin{proposition} \label{prop:7}
Loss function $\mathcal{F}_{\alpha}(W)$ in equation \eqref{eq:Loss} has the explicit form:
\begin{equation*}
   \mathcal{F}_{\alpha}(W)=\frac{1}{2}\sum_{i,j=1}^{k}f_{\alpha}( w_i,  w_j)-\sum_{i=1}^{k}\sum_{j=1}^{s}f_{\alpha}( w_i,v_j)+\frac{1}{2}\sum_{i,j=1}^{s}f_{\alpha}(v_i, v_j),
\end{equation*}
where $f_{\alpha}$ is given by formula \eqref{eq: explicit_f}.
\end{proposition}
\begin{proof}
\begin{align*}
\mathcal{F}_{\alpha}(W)&=\frac{1}{2}\mathbb{E}_{x\sim \mathcal N(0,\mathbb{I}_k)}\Big[\Big(\sum_{i=1}^{k}\sigma_{\alpha}({ w_i}^T x)\Big)^2-2\sum_{i=1}^{k}\sum_{j=1}^{s}\sigma_{\alpha}({ w_i}^Tx)\sigma_{\alpha}({v_j}^T x)+\Big(\sum_{i=1}^{s}\sigma_{\alpha}({v_i}^Tx)\Big)^2\Big]\\
&=\frac{1}{2}\mathbb{E}_{x\sim \mathcal N(0,\mathbb{I}_k)}\Big[\sum_{i=1}^{k}\sum_{j=1}^{k}\sigma_{\alpha}({w_i}^Tx)\sigma_{\alpha}({w_j}^Tx)-2\sum_{i=1}^{k}\sum_{j=1}^{s}\sigma_{\alpha}({w_i}^T x)\sigma_{\alpha}({ v_j}^Tx)+\sum_{i=1}^{s}\sum_{j=1}^{s}\sigma_{\alpha}({ v_i}^Tx)\sigma_{\alpha}({v_j}^T x)\Big]\\
\end{align*}
Notice that expectation operator $\mathbb E:\mathbb R \to \mathbb{R}$ is linear, therefore
\begin{align*}
\mathcal{F}_{\alpha}(W)
&=\frac{1}{2}\sum_{i=1}^{k}\sum_{j=1}^{k}\mathbb{E}_{x\sim \mathcal N(0,\mathbb{I}_k)} \Big(\sigma_{\alpha}({ w_i}^Tx)\sigma_{\alpha}({w_j}^T x)\Big)-\sum_{i=1}^{k}\sum_{j=1}^{s}\mathbb{E}_{x\sim \mathcal N(0,\mathbb{I}_k)} \Big(\sigma_{\alpha}({w_i}^T x)\sigma_{\alpha}({v_j}^T x)\Big)\\
&+\frac{1}{2}\sum_{i=1}^{s}\sum_{j=1}^{s}\mathbb{E}_{x\sim \mathcal N(0,\mathbb{I}_k)} \Big(\sigma_{\alpha}({ v_i}^T x)\sigma_{\alpha}({ v_j}^T x)\Big)\\
&=\frac{1}{2}\sum_{i,j=1}^{k}f_{\alpha}( w_i,  w_j)-\sum_{i=1}^{k}\sum_{j=1}^{s}f_{\alpha}( w_i,v_j)+\frac{1}{2}\sum_{i,j=1}^{s}f_{\alpha}( v_i, v_j)
\end{align*}
\end{proof}
\subsection{Explicit Form of Gradient of Loss Function}\label{ap:derivation_gls}
For completeness, we summarize here the expressions for the loss function, its gradient, and related derivations following 
\begin{lemma}
    Let $\varphi_{\alpha} ( w):=f_{\alpha}( w,  v)$ in equation \eqref{eq: explicit_f}, then $\nabla_{ w} \varphi_{\alpha}:\mathbb{R}^k \setminus\{0\}\to \mathbb{R}^k$ is given by
\begin{equation}
    \nabla_{ w} \varphi_{\alpha}( w)=\frac{\sin \theta\| v\|}{2\pi \| w\|}  w+
    \frac{\pi-\theta}{2\pi} v
\end{equation}
\end{lemma}
\begin{proof}
Notice that
$\theta=\cos^{-1}\frac{ w\cdot  v}{\| w\|\| v\|}$, let $b=\frac{w\cdot  v}{\|w\|\| v\|},$ one has  
\begin{align*}
    \nabla_{w}\theta&=-\frac{1}{\sqrt{1-b^2}}\frac{\partial b}{\partial  w}\\&=\frac{1}{\sqrt{1-b^2}}\Big(\frac
    { w\cdot  v\cdot  w}{\| w\|^3\| v\|}-\frac{ v}{\|w\|\| v\|}\Big)
\end{align*}
and \begin{align*}
\nabla_{ w} \Big(\frac{\sin \theta\| v\|}{2\pi \|w\|}\Big)&=\frac{\| v\|}{2\pi}\nabla_{w}\|w\|=\frac{\| v\|}{2\pi}\nabla_{w}\sqrt{ w\cdot  w}\\&=\frac{2\| v\|w}{4\pi\|w\|}=\frac{\| v\|}{2\pi\|w\|} w,
\end{align*}
Therefore,
  \begin{align*}
    \nabla_{ w} \varphi_{\alpha}( w)&=\nabla_{w} \Big(\frac{\sin \theta\| v\|}{2\pi \|w\|}\Big)\Big(\sin\theta+(\pi-\theta)\cos\theta\Big)+ \frac{\|w\|\| v\|}{2\pi}\nabla_{ w}\Big(\sin\theta+(\pi-\theta)\cos\theta\Big)\\
&=\frac{\| v\|}{2\pi\| w\|} w\Big(\sin\theta+(\pi-\theta)\cos\theta\Big)+ \frac{\|w\|\| v\|}{2\pi}\nabla_{ w}\sin\theta+ \frac{\|w\|\| v\|}{2\pi}\nabla_{w}\Big((\pi-\theta)\cos\theta\Big)\\
&=\frac{\|v\|}{2\pi\| w\|}w\Big(\sin\theta+(\pi-\theta)\cos\theta\Big)+\frac{\| w\|\| v\|}{2\pi}\cos\theta\nabla_{w}\theta\\&+\frac{\| w\|\| v\|}{2\pi}\cos\theta\nabla_{ w}(\pi-\theta)+\frac{\|w\|\| v\|}{2\pi}(\pi-\theta)\nabla_{ w}\cos\theta\\
&=\frac{\|v\|}{2\pi\|w\|}w\Big(\sin\theta+(\pi-\theta)\cos\theta\Big)-\frac{\| w\|\| v\|}{2\pi}(\pi-\theta)\sin\theta\nabla_{ w}\theta\\
&=\frac{\sin\theta\| v\|}{2\pi\|w\|}w+\frac{\| v\| w(\pi-\theta)}{2\pi\| w\|}\cos\theta-\frac{\| v\|}{2\pi\| w\|}(\pi-\theta)\sin\theta\nabla_{ w}\theta\\
&=\frac{\sin\theta\| v\|}{2\pi\| w\|} w+\frac{\|v\| w(\pi-\theta)}{2\pi\|w\|}\cos\theta+\frac{(\pi-\theta)\sin\theta}{\sqrt{1-b^2}2\pi}\Big(-\frac{ w\cdot  v \cdot  w}{\| w\|^2}+v\Big)\\
&=\frac{\sin\theta\| v\|}{2\pi\|w\|}w+\frac{\|v\| w(\pi-\theta)}{2\pi\|w\|}\cos\theta-\frac{(\pi-\theta)}{2\pi}\frac{w\cdot  v\cdot  w}{\| w\|^2}+\frac{\pi-\theta}{2\pi} v\\
&=\frac{\sin\theta\| v\|}{2\pi\|w\|}w+\frac{\|v\| w(\pi-\theta)}{2\pi\|w\|}\cos\theta-\frac{(\pi-\theta)}{2\pi}\frac{w\cdot  v\cdot  w\| v\|}{\| w\|\| v\|\|w\|}+\frac{\pi-\theta}{2\pi} v\\
&=\frac{\sin\theta\| v\|}{2\pi\|w\|}w+\frac{\|f v\| w(\pi-\theta)}{2\pi\|w\|}\cos\theta-\frac{(\pi-\theta)}{2\pi}\cos\theta \frac{\|v\| w}{\| w\|}+\frac{\pi-\theta}{2\pi} v\\
&=\frac{\sin\theta\| v\|}{2\pi\|w\|} w+\frac{\pi-\theta}{2\pi} v
\end{align*}
\end{proof}\vskip0.6cm
\noi Define $\Omega:=\{W\in \mathbb R^{k^2}:u_i \neq 0,i\in\{1,\cdots,k\}\},$ assume  $V^o=\operatorname{vec}(\mathbb I_k).$ 
\begin{proposition} For $W\in \mathbb R^{k^2},$ the gradient
$\nabla_{W}\mathcal F_{\alpha}:\Omega\to \mathbb R^{k^2}$ is given by 
\[
\nabla_{W}F_{\alpha}(W)=
{\begin{bmatrix}
   \nabla_{ w^1}\mathcal F_{\alpha}(W),\nabla_{ w^2}\mathcal F_{\alpha}(W),\cdots, \nabla_{ w^k}\mathcal F_{\alpha}(W) 
\end{bmatrix}}^T,
\]
where
\begin{align*}\label{eq:explicit_lg}
\nabla_{w_i}\mathcal F_{\alpha}(W)&=\frac{1}{2\pi}\sum_{j=1}^{k}\Big(\frac{\| w_j\|\sin\theta_{ij}}{\| w_i\|} w_i+(\pi-\theta_{ij}) w_j\Big)\notag\\
&-\frac{1}{2\pi}\sum_{j=1}^{k}\Big(\frac{\sin\tilde{\theta}_{ij}}{\| w_i\|} w_i+(\pi-\tilde{\theta}_{ij}) v_j\Big), 
\end{align*}
$({\theta}_{ij}=\cos^{-1}\frac{ w_i\cdot  w_j}{\| w_i\|\| w_j\|}\text{ and } \tilde{\theta}_{ij}=\cos^{-1}\frac{\mathbf w_i\cdot  v_j}{\| w_i\|\| v_j\|})$
\end{proposition}
\begin{proof}
    Given explicit form of loss function \eqref{eq:explicit_l}, one has
\begin{align*}
 \nabla_{w_i}\mathcal F_{\alpha}(W)&=\frac{1}{2}\sum_{j=1}^{k}\nabla_{ w_i}f_{\alpha}( w_i. w_j)+   \frac{1}{2}\sum_{j=1}^{k}\nabla_{ w_i}f_{\alpha}(w_j. w_i)-\sum_{j=1}^{k}\nabla_{ w_i}f_{\alpha}( w_i. v_j)\\
&=\sum_{j=1}^{k}\nabla_{ w_i}f_{\alpha}(w_i. w_j)-\sum_{j=1}^{k}\nabla_{ w_i}f_{\alpha}( w_i. v_j)\\
&=\frac{1}{2\pi}\sum_{j=1}^{k}\Big(\frac{\| w_j\|\sin\theta_{ij}}{\| w_i\|} w_i+(\pi-\theta_{ij}) w_j\Big)-\frac{1}{2\pi}\sum_{j=1}^{k}\Big(\frac{\sin\tilde{\theta}_{ij}}{\|w_i\|} w_i+(\pi-\tilde{\theta}_{ij}) v_j\Big)
\end{align*}
\end{proof}
\subsection{Derivation of Hessian \texorpdfstring{$\nabla^2_{u}\mathcal F_{\alpha}(u)$}{Hessian}}\label{app:hessian}
Let $x,y\in \mathbb R^k$ be two non-parallel vectors and $\hat{x}=\frac{x}{\|x\|},\; \hat{y}=\frac{y}{\|y\|}$. Let $\theta_{xy}\in (0,\pi)$ denotes the angle between them. Define
\begin{equation}\label{eq: Phi}
\Phi(x,y):=\|y\|\sin\theta_{xy}\hat{x}-\theta_{xy}y
\end{equation}
Then gradient formula \eqref{eq:lg} can be written as
\[
\nabla_{u_i}\mathcal F_{\alpha}(u)=\frac{\alpha}{2\pi}\sum_{j=1}^{k}\Phi(u_i,u_j)
-\frac{\alpha}{2\pi}\sum_{j=1}^{k}\Phi(u_i,v_j)+\frac{1}{2}\sum_{j=1}^{k}(u_j-v_j).
\]
To obtain $\nabla^2_u\mathcal F_{\alpha}(u)\in\mathbb R^{k^2}\times \mathbb R^{k^2},$ we carry out the following calculations:\vskip 0.3cm
\noindent 1)
Put $n_{xy}:=\hat{x}-\cos\theta_{xy}\hat{y},\quad \hat{n}_{xy}:=\frac{n_{xy}}{\|n_{xy}\|}.$
Then $\|n_{xy}\|=\sin\theta_{xy}.$ Indeed, 
\begin{align*}
\|n_{xy}\|^2&=n_{xy}^Tn_{xy}=\Big(\hat{x}^T-\cos\theta_{xy}\hat{y}^T\Big)\Big(\hat{x}-\cos\theta_{xy}\hat{y}\Big)\\
&=\hat{x}^T\hat{x}-2\cos\theta_{xy}\hat{x}^T\hat{y}+\cos^2\theta_{xy}\hat{y}^T\hat{y}=1-\cos^2\theta_{xy}=\sin^2\theta_{xy}.
\end{align*}
Notice $n_{xy}\perp \hat{y},\; n_{yx}\perp \hat{x}.$\vskip 0.3cm
\noindent 2) We next compute $\frac{d\hat{x}}{dx}.$ Let $r=\|x\|=(x^Tx)^{1/2}.$ So $\hat{x}=\frac{x}{r}$ and $d\hat{x}=\frac{1}{r}dx+xd(\frac{1}{r}),$
where $d(\frac{1}{r})=-r^{-2}dr=-r^{-3}x^Tdx.$ Therefore, $d\hat{x}=\frac{1}{r}dx-\frac{xx^T}{r^3}dx$, i.e.
\[
\frac{d\hat{x}}{dx}=\frac{1}{\|x\|}(\mathbb I_k-\hat{x}\hat{x}^T).
\]
\noindent 3) Next compute $\nabla_x\theta_{xy},\; \nabla_y\theta_{xy}.$ Let $c:=\cos\theta_{xy}=\hat{x}\cdot\hat{y},$ then $\theta_{xy}=\arccos c.$ So
\[
\nabla_x\theta_{xy}=-\frac{\nabla_x c}{\sqrt{1-c^2}}=-\frac{\nabla_x c}{\sin\theta_{xy}},
\]
where
\begin{align*}
\nabla_x c&=\nabla_x(\hat{x}\cdot \hat{y})=(\nabla_x\hat{x})\cdot \hat{y}\\
&=\frac{1}{\|x\|}(\mathbb I_k-\hat{x}\hat{x}^T)\cdot \hat{y}=\frac{1}{\|x\|}(\hat{y}-(\hat{x}\cdot \hat{y})\hat{x})=\frac{n_{yx}}{\|x\|}.
\end{align*}
Therefore, $\nabla_x\theta_{xy}=-\frac{\hat{n}_{yx}}{\|x\|},$ and similarly $\nabla_y\theta_{xy}=-\frac{\hat{n}_{xy}}{\|y\|}.$ \vskip 0.3cm
\noindent 4) Compute $\nabla_x\Big(\sin\theta_{xy}\Big),\; \nabla_y\Big(\sin\theta_{xy}\Big).$ Obviously, one has
\begin{align*}
&\nabla_x\Big(\sin\theta_{xy}\Big)=\cos\theta_{xy}\nabla_x(\theta_{xy})=-\frac{\cos\theta_{xy}\hat{n}_{yx}}{\|x\|}\\
&\nabla_y\Big(\sin\theta_{xy}\Big)=\cos\theta_{xy}\nabla_y(\theta_{xy})=-\frac{\cos\theta_{xy}\hat{n}_{xy}}{\|y\|}.
\end{align*}
\[
\nabla_x\theta_{xy}=\frac{\hat{n}_{xy}}{\|x\|},\quad \nabla_y\theta_{xy}=\frac{\hat{n}_{yx}}{\|y\|}.
\]
One also has $D_x\hat{x}=\frac{1}{\|x\|}(I-\hat{x}\hat{x}^T)$ and 
\[
\nabla_x(\sin\theta_{xy})=\cos\theta_{xy}\nabla_x\theta_{xy}=\frac{\cos\theta_{xy}}{\|x\|}\hat{n}_{xy},\;\quad \nabla_y(\sin\theta_{xy})=\cos\theta_{xy}\nabla_y\theta_{xy}=\frac{\cos\theta_{xy}}{\|y\|}\hat{n}_{yx}.
\]
\noindent 5) Recall $\Phi(x,y)$ in formula \eqref{eq: Phi}, we then define 
\begin{align*}
h_1(x,y):&=\frac{1}{2\pi}\Phi_x(x,y)=\frac{1}{2\pi}\Big(\|y\|\cos\theta_{xy}\cdot \nabla_x\theta_{xy}\;\hat{x}^T+\|y\|\sin\theta_{xy}\nabla_x\hat{x}-\nabla_x\theta_{xy}\; y^T\Big)
\\&=\frac{\sin\theta_{xy} \|y\|}{2\pi \|x\|}\Big(\mathbb I_k-\frac{xx^T}{\|x\|^2}+\hat{n}_{yx}\hat{n}_{yx}^T\Big)\\
    h_2(x,y):&=\frac{1}{2\pi}\Phi_y(x,y)=\frac{1}{2\pi}\Big(-\theta_{xy}\mathbb I_k+\frac{\hat{n}_{xy}y^T}{\|y\|}+\frac{\hat{n}_{yx}x^T}{\|x\|}\Big)
\end{align*}
Then general form of Hessian $\nabla^2_u\mathcal (u)$ in equation \eqref{eq:hessian} can be easily derived.

\subsection{Spectrum of Hessian \texorpdfstring{$\mathcal A_{\alpha}(v^o)$}{A(vo)}.}\label{ap:spectrum} Let $X=[X_1,\cdots, X_k]$ be $k\times k$ matrix where $X_j\in \mathbb R^k,\; j=1,\cdots,k.$ Define block operator $\mathcal L:\mathbb R^{k\times k}\to \mathbb R^{k\times k}$ given by
\[
(\mathcal LX)_i:=\sum_{j=1}^k A_{ij}X_j,\quad j=1,\cdots,k.
\]
Set $a=\frac{1}{2}-\frac{\alpha}{4},\; b=\frac{\alpha}{4},\;c=\frac{\alpha}{2\pi},$ and $J=\mathbf{1}\mathbf{1}^T,$  where $\mathbf{1}=(1,\cdots,1)^T\in\mathbb R^k.$ By equation \eqref{eq:spe_I}, one can easily derive
\begin{align}\label{eq:spectrum_com}
(\mathcal LX)_i&=A_{ii}X_i+\sum_{j\neq i}A_{ij}X_j=\frac{1}{2}X_i+\sum_{j\neq i}\Big(\frac{1}{2}-\frac{\alpha}{4}\Big)X_j+\sum_{j\neq i}\frac{\alpha}{2\pi}\Big(E_{ij}+E_{ji}\Big)X_j\notag\\
    &=\frac{1}{2}X_i+a\sum_{j\neq i}X_j+c\sum_{j\neq i}\Big(E_{ij}+E_{ji}\Big)X_j.
\end{align}
Notice, $\frac{1}{2}X_i+a\sum_{j\neq i}X_j=\frac{1}{2}X_i+a\Big(\sum_{j=1}^kX_j-X_i\Big),$ which can be written as $aX\mathbf{1}+bX_i.$ This also represents the $i$-th column of matrix
\[
X(aJ+b\mathbb I_k).
\]
On the other hand,
\begin{align*}
    \sum_{j\neq i}cE_{ij}X_j&=\sum_{j\neq i}ce_i {e_j}^T X_j=ce_i\sum_{j\neq i}{e_j}^TX_j=ce_i\sum_{j\neq i}X_{jj}=ce_i\Big(\operatorname{tr}(X)-X_{ii}\Big),\\
\sum_{j\neq i}cE_{ji}X_j&=c\sum_{j\neq i}e_j \Big({e_i}^T X_j\Big)=c\sum_{j\neq i}e_jX_{ij}=c\Big(\sum_{j=1}^k e_jX_{ij}-e_iX_{ii}\Big)=c\Big(X^Te_i-e_iX_{ii}\Big),    
\end{align*}
where $X_{ij}$ denotes the component of $i$-th row, $j$-th column of matrix $X.$ Therefore, $\sum_{j\neq i}\Big(cE_{ij}X_j+cE_{ji}X_j\Big)=c\Big[(X^T)e_i+e_i\Big(\operatorname{tr}(X)-2X_{ii}\Big)\Big],$ i.e. the $i$-th column of 
\[
c\Big(X^T+\operatorname{tr}(X)\mathbb I_k-2\operatorname{diag(X)}\Big).
\]
Therefore, $    T(X)=X\Big(aJ+b\mathbb I_k\Big)+c\Big(X^T+\operatorname{tr}(X)\mathbb I_k-2\operatorname{diag}(X)\Big).$

\section{Irreducible Representation of \texorpdfstring{$S_k$}{Sk}}\label{app:A2}
For the reader’s convenience, we provide a detailed derivation of decomposition \eqref{eq: decomposition} in this section. Since this is a standard result, we refer the readers to \cite{fulton2013representation} Chapter 4, for a comprehensive discussion. Given the well-known fact that (or see Example 4.6 in \cite{fulton2013representation}) for general $k,$ the standard representation $V_{\perp}$ of $V:=\mathbb R^{k^2}$ corresponds to partition $k=(k-1)+1,$ i.e. 
\[
V_{\perp}\cong S^{(k-1,1)}.
\]
We next introduce the Frobenius Character Formula
\begin{equation}\label{eq:frobenius}
\chi_\eta(C_i)
= 
\bigl[
\Delta(x)\,
\prod_{j} P_j(x)^{\,i_j}
\bigr]_{(l_1,\dots,l_r)},
\end{equation}
where
\begin{enumerate}
\item $\chi_\eta$ is the irreducible character of $S_k$ corresponding to the partition $\eta$.
\item $C_i$ is the conjugacy class determined by 
      \[
      i = (i_1, i_2, \dots,i_r), \qquad \sum_j j\,i_j = k,
      \]
      i.e.\ there are $i_1$ 1-cycles, $i_2$ 2-cycles, etc.
\item $P_j(x) = x_1^j + x_2^j + \cdots + x_r^j$ are the power sums in $r$ independent variables.
\item $\Delta(x) = \prod_{i<j}(x_i - x_j)$ is the Vandermonde determinant.
\item The bracket notation $[f(x)]_{(l_1,\dots,l_r)}$
      means take the coefficient of $x_1^{l_1}\cdots x_r^{l_r}$ in the expansion of $f(x)$.
\item The integers $l_j$ are defined from the partition $\eta = (\eta_1 \ge \cdots \ge \eta_r)$ by
      \[
      l_j = \eta_j + r - j.
      \]
\end{enumerate}
One can use this formula to explicitly compute the characters for several fundamental partitions.
\begin{proposition}\label{prop: frobenius} Let $k = \sum_j j\,i_j$, then we have
  \begin{enumerate}
      \item[(a)] If $\eta = (k),$ then $\chi_{(k)}(C_i) = 1.$
      \item[(b)] If $\eta = (k-1,1),$ then $\chi_{(k-1,1)}(C_i) = i_1 - 1.$ 
      \item[(c)] If $\eta= (k-2,1,1)$, then $\chi_{(k-2,1,1)}(C_i)=\frac{1}{2}(i_1-1)(i_1-2) - i_2.$
      \item[(d)] If $\eta = (k-2,2)$,
$\chi_{(k-2,2)}(C_i)
=
\frac{1}{2}(i_1-1)(i_1-2) + i_2 - 1.$
  \end{enumerate}  
\end{proposition}
\begin{proof} By applying Frobenius Formula \eqref{eq:frobenius}, we have
\begin{enumerate}
\item[(a)] For $\eta = (k),$ we have $r=1$, $l_1 = k$ and $\Delta(x)=1$, $P_j(x)=x_1^j$.
Thus, $\chi_{(k)}(C_i) = 1$ for all $i$.
\item[(b)] For $\eta = (k-1,1)$. Then $r=2$ and $l=(k,1).$ Thus,
\[
\chi_{(k-1,1)}(C_i)
=
\Bigl[
(x_1-x_2)\,
\prod_j (x_1^j + x_2^j)^{i_j}
\Bigr]_{x_1^{k}x_2^{1}}.
\]
Let $\Phi(x_1,x_2):=\Bigr[\prod_j (x_1^j + x_2^j)^{i_j}
\Bigr]_{x_1^{k}x_2^{1}},$ then $
\chi_{(k-1,1)}(C_i)=\Bigr[x_1\Phi\Bigr]_{x_1^{k}x_2^{1}}-\Bigr[x_2\Phi\Bigr]_{x_1^{k}x_2^{1}}$ which is equivalent to 
\[
\Bigr[\Phi\Bigr]_{x_1^{k-1}x_2^{1}}-\Bigr[\Phi\Bigr]_{x_1^{k}x_2^{0}}.
\]
For $\Bigr[\Phi\Bigr]_{x_1^{k}x_2^{0}},$ one has for each $j,$ elements of $(x_1^j+x_2^j)^{i_j}$ are of the form 
\[
(x_1^j)^{a_j}(x_2^j)^{b_j},\;\;a_j+b_j=i_j.
\]
Therefore, 
\[
\Bigr[\Phi\Bigr]_{x_1^{k}x_2^{0}}=\Bigr[\prod_{j}\binom{i_j}{a_j}x_1^{\sum_j ja_j}x_2^{\sum_j jb_j}\Bigr]_{x_1^kx_2^0},
\]
which implies
\[
\sum_j ja_j = k,\quad
\sum_j jb_j=0.
\]
i.e. $b_j=0$ and $a_j=i_j,$ so $\Bigr[\Phi\Bigr]_{x_1^{k}x_2^{0}}=\binom{i_j}{i_j}=1.$ Similarly, $\Bigr[\Phi\Bigr]_{x_1^{k-1}x_2^{1}}=i_1.$
Therefore,
\[
\chi_{(k-1,1)}(C_i) = i_1 - 1.
\]
\item[(c)] For $\eta= (k-2,1,1)$, we have $r=3$ with $l = (k,2,1)$, and
\[
\chi_{(k-2,1,1)}(C_i)
=
\Bigl[
(x_1-x_2)(x_1-x_3)(x_2-x_3)
\prod_j (x_1^j + x_2^j + x_3^j)^{i_j}
\Bigr]_{x_1^{k}x_2^{2}x_3^{1}}.
\]
\item[(d)] For $\eta = (k-2,2),$ similarly, we have $r=2$ with $l = (k-1,2)$.
\end{enumerate}
\end{proof}
\vskip0.3cm
\noindent Therefore, for $g\in S_k,$ one has $\chi_{S^{(k-1,1)}}(g)=i_1(g)-1,$ where $i_1(g)$ denotes the number of 1-cycles in $g,$ i.e., the number of fixed points of $g.$ Moreover, it is straightforward to verify that 
\[
i_1(g^2)=i_1(g)+2i_2(g),
\]
where $i_2(g)$ is the number of 2-cycles in $g$.
On the other hand, using the standard character identities
\begin{align*}
\chi_{\operatorname{Sym}^2(V_{\perp})}(g)=\frac{1}{2}\Big(\chi_{V_{\perp}}(g)^2+\chi_{V_{\perp}}(g^2)\Big),\quad
\chi_{\wedge^2(V_{\perp})}(g)=\frac{1}{2}\Big(\chi_{V_{\perp}}(g)^2-\chi_{V_{\perp}}(g^2)\Big),
\end{align*}
we obtain
\[
\chi_{\operatorname{Sym}^2(V_{\perp})}=\frac{1}{2}\Big((i_1-1)^2+i_1+2i_2-1\Big),\;\;
    \chi_{\wedge^2(V_{\perp})}=\frac{1}{2}\Big((i_1-1)^2-(i_1+2i_2-1)\Big).
\]
By comparison with Proposition\eqref{prop: frobenius}, we derive the following decompositions
\begin{align*}
&\operatorname{Sym}^2(V_{\perp})\cong S^{(k)}\oplus S^{(k-1,1)}\oplus S^{(k-2,2)}\\
&\wedge^2(V_{\perp})\cong S^{(k-2,1,1)}.\notag 
\end{align*}
Equation \eqref{eq: decomposition} follows.
\section{Euler and Burnside Rings\label{app:B}}
In this section we assume that $G$ stands for a compact Lie group and we
denote by $\Phi(G)$ the set of all conjugacy classes $(H)$ of closed subgroups
$H$ of $G$. For any $(H)\in\Phi(G)$ we denote by $N(H)$ the normalizer of $H$
and by $W(H):=N(H)/H$ the Weyl's group of $H$. \vskip.3cm

Notice that $\Phi(G)$ admits a natural order relation given by
\begin{equation}
\label{eq:order}(K)\le(H) \;\; \Leftrightarrow\;\; \exists_{g\in G}\; \;
gKg^{-1}\subset H, \;\; \text{ for }\;\; (K),\,(H)\in\Phi(G).
\end{equation}
Moreover, we define for $n=0,1,2,\dots$ the following subsets $\Phi_{n}(G)$ of
$\Phi(G)$
\[
\Phi_{n}(G):=\{ (H)\in\Phi(G): \dim W(H)=n\}.
\]
\vskip.3cm \noi Let  
$U(G)={\mathbb{Z}}[\Phi(G)]$ be the free ${\mathbb{Z}}$-module generated by $\Phi(G)$, then an element $a\in U(G)$ is represented as
\begin{equation}
\label{eq:sum-a}a=\sum_{(L)\in\Phi(G)} n_{L} \, (L), \quad n_{L}\in
{\mathbb{Z}},
\end{equation}
where the integers $n_{L}=0$ except for a finite number of elements
$(L)\in\Phi(G)$. For such element $a\in U(G)$ and $(H)\in\Phi(G)$, we will
also use the notation
\begin{equation}
\label{eq:coeff-H}\text{coeff}^{H}(a)=n_{H},
\end{equation}
i.e. $n_{H}$ is the coefficient in \eqref{eq:sum-a} standing by $(H)$. 
\vskip.3cm

\begin{definition}
\textrm{\label{def:EulerRing} (cf. \cite{tD}) Define the
\textit{multiplication} on $U(G)$ as follows: for generators $(H)$,
$(K)\in\Phi(G)$ put:
\begin{equation}
(H)\ast(K)=\sum_{(L)\in\Phi(G)}n_{L}(L),\label{eq:Euler-mult}%
\end{equation}
where
\begin{equation}
n_{L}:=\chi_{c}((G/H\times G/K)_{L}/N(L)),\label{eq:Euler-coeff}%
\end{equation}
Note,  $(G/H\times G/K)_{L}$ denotes the set of elements in $G/H\times G/K$ that are
fixed exactly by $L$. $\chi_c(\cdot)$ denotes the Euler Characteristic (For its precise definition, see Section 3 of \cite{DK} and \cite{Spa}).} Moreover, the multiplication is extended linearly to the entire Euler ring 
$U(G)$. Then the free ${\mathbb{Z}}$-module $U(G)$ associated with
multiplication (\ref{eq:Euler-mult}) is called the \textit{Euler ring} of $G$.
\end{definition}

\vskip.3cm It is easy to notice that $(G)$ is the unit element in $U(G)$, i.e.
$(G)*a=a$ for all $a\in U(G)$. \vskip.3cm

\begin{lemma}
\label{lem:inv-1} Assume that $a\in U(G)$ is an invertible element and
$(H)\in\Phi(G)$. Then
\[
\text{\text{\rm coeff}}^{H}((H)*a)\not =0.
\]

\end{lemma}

\begin{proof}
Suppose that
\[
a=\sum_{(L)\in\Phi(a)} n_{L} \, (L).
\]
Then
\[
(H)*a=\sum_{(K)\in\Phi((H)*a)} m_{K}\, (K), \text{ and formula
(\ref{eq:Euler-mult}) implies that} \;\; (H)\ge(K).
\]
Assume for contradiction that $(H)>(K)$ for all $(K)\in\Phi((H)*a)$. Then, by
exactly the same argument we have
\[
(H)*a*a^{-1}=\sum_{(L)\in\Phi((H)*a*a^{-1})} n_{L} \, (L), \quad\text{ where }
(H)>(L),
\]
which is a contradiction with the fact that
\[
(H)*a*a^{-1}=(H)*(G)=(H).
\]
\end{proof}

\vskip.3cm Take $\Phi_{0}(G)=\{(H)\in\Phi(G):$ \textrm{dim\thinspace
}$W(H)=0\}$ and denote by $A(G)={\mathbb{Z}}[\Phi_{0}(G)]$ a free
${\mathbb{Z}}$-module with basis $\Phi_{0}(G).$ Define multiplication on
$A(G)$ by restricting multiplication from $U(G)$ to $A(G)$, i.e. for $(H)$,
$(K)\in\Phi_{0}(G)$ let
\begin{equation}
(H)\cdot(K)=\sum_{(L)}n_{L}(L),\qquad(H),(K),(L)\in\Phi_{0}(G),\text{ where}
\label{eq:multBurnside}%
\end{equation}
\begin{equation}
n_{L}=\chi((G/H\times G/K)_{L}/N(L))=|(G/H\times G/K)_{L}/N(L)|
\label{eq:CoeffBurnside}%
\end{equation}
($\chi$ here denotes the usual Euler characteristic). Then $A(G)$ with
multiplication \eqref{eq:multBurnside} becomes a ring which is called the
\textit{Burnside ring} of $G$. As it can be shown, the coefficients
\eqref{eq:CoeffBurnside} can be found using the following
recursive formula:
\begin{equation}
n_{L}=\frac{n(L,K)|W(K)|n(L,H)|W(H)|-\sum_{(\widetilde{L})>(L)}n(L,\widetilde
{L})n_{\widetilde{L}}|W(\widetilde{L})|}{|W(L)|}, \label{eq:rec-coef}%
\end{equation}
where $(H),$ $(K),$ $(L)$ and $(\widetilde{L})$ are taken from $\Phi_{0}
(G),$ and \begin{align*}
&
N_{G}(L,H)=\Big\{
g\in G:gLg^{-1} \subset H\Big\} , \\
&N_{G}(L,H)/H=\Big\{
Hg: g\in N_{G}(L,H)\Big\}\\
& n(L,H)=\Big|\frac{N(L,H)}{N(H)}\Big|
\end{align*}
\vskip.3cm Observe that although $A(G)$ is clearly a ${\mathbb{Z}}$-submodule
of $U(G)$, in general, it is \textbf{not} a subring of $U(G)$.

\vskip.3cm Define $\pi_{0}:U(G)\rightarrow A(G)$ as follows: for $(H)\in
\Phi(G)$ let
\begin{equation}
\pi_{0}((H))=
\begin{cases}
(H) & \text{ if }\;(H)\in\Phi_{0}(G),\\
0 & \text{ otherwise.}%
\end{cases}
\label{eq:pi_0-homomorphism}%
\end{equation}
The map $\pi_{0}$ defined by (\ref{eq:pi_0-homomorphism}) is a ring
homomorphism {(cf. \cite{BKR})}, i.e.
\[
\pi_{0}((H)\ast(K))=\pi_{0}((H))\cdot\pi_{0}((K)),\qquad(H),(K)\in\Phi(G).
\]
The following well-known result (cf. \cite{tD}, Proposition 1.14, page 231)
shows a difference between the generators $(H)$ of $U(G)$ and $A(G)$.

\begin{proposition}
\label{pro:nilp-elements} Let $(H)\in\Phi_{n}(G)$.

\begin{itemize}
\item[(i)] If $n>0$, then $(H)^{k}=0$ in $U(G)$ for some $k\in\mathbb{N}$,
i.e. $(H)$ is a nilpotent element in $U(G)$;

\item[(ii)] If $n=0$, then $(H)^{k}\not =0$ for all $k\in{\mathbb{N}}$.
\end{itemize}
\end{proposition}

\begin{corollary}
    If $\alpha=n_1(L_1)+n_2(L_2)+\cdots +n_k(L_k),$ where $\dim W(L_j)\ge 0,$ then there exists $n\in \mathbb N$ s.t. $\alpha^n=0.$
\end{corollary}
\begin{proof}
    By induction w.r.t. $k\in \mathbb N$, clearly for $k=1,$ it is exactly the statement of proposition\eqref{pro:nilp-elements}. Suppose that the statement is true for $k\geq 1,$ and will show that it is also true for $k+1.$ Indeed, we have
    \[
\alpha^1=n_1(L_1)+n_2(L_2)+\cdots n_k(L_k)+n_{k+1}(L_{k+1})
    =\alpha+n_{k+1}(L_{k+1}),
    \]
so \[
\alpha^m=\sum_{l=0}^{m}C_m^l \alpha^l n_k^{m-l}(L_{k+1})^{m-l}.
\]
Let $k$ be given by Proposition \eqref{pro:nilp-elements}, for $L:=l_{k+1},$ then for $m\ge n+k,$ one has 
\[
\alpha^l (L)^{m-l}=0
\]
    \end{proof}
\vskip.3cm 

\black
\section{Properties of \texorpdfstring{$G$}{G}-Equivariant Gradient Degree}\label{app:C}

In what follows, we assume that $G$ is a compact Lie group.

\subsection{Brouwer $G$-Equivariant Degree}
Assume that \texorpdfstring{$V$}{V} is an orthogonal $G$-representation and $\Omega\subset V$ an
open bounded $G$-invariant set. A $G$-equivariant (continuous) map $f:V\to V$
is called \textit{$\Omega$-admissible} if $f(x)\neq0$ for any $x\in
\partial\Omega$; in such a case, the pair $(f,\Omega)$ is called
\textit{$G$-admissible}. Denote by $\mathcal{M}^{G}(V,V)$ the set of all such
admissible $G$-pairs, and put $\mathcal{M}^{G}:=\bigcup_{V}\mathcal{M}
^{G}(V,V)$, where the union is taken for all orthogonal $G$-representations
$V$. We have the following result:

\begin{definition}
\label{thm:GpropDeg} There exists a unique map $G$-$\deg:\mathcal{M}^{G}\to
A(G)$, which assigns to every admissible $G$-pair $(f,\Omega)$ an element
$G\mbox{\rm -}\deg(f,\Omega)\in A(G)$, called the \textit{$G$-equivariant
degree} (or simply \textit{$G$-degree}) of $f$ on $\Omega$:
\begin{equation}
\label{eq:G-deg0}G\mbox{\rm -}\deg(f,\Omega)=\sum_{(H_{i})\in\Phi_{0}
(G)}n_{H_{i}}(H_{i})= n_{H_{1}}(H_{1})+\dots+n_{H_{m}}(H_{m}).
\end{equation}
It satisfies the properties of additivity, homotopy, normalization, as well as existence, product, suspension, recurrence formula, etc. (see \cite{balanov2025degree} for details).
We call $G\mbox{\rm -}\deg(f,\Omega)$ the \textit{$G$-equivariant
degree} (or simply \textit{$G$-degree}) of $f$ on $\Omega$.
\end{definition}\vskip0.3cm
\begin{definition} The Brouwer $G$-equivariant degree
\begin{equation} \label{eq:basicdeg}
\deg_{\cV_i} := G\text{-}\deg(-\id, B(\cV_i)),
\end{equation}
is called the {\it $\cV_i$-basic degree} (or simply {\it basic degree}), and it can be computed by: $\deg_{\mathcal{V} _{i}}=\sum_{(K)}n_{K}(K),$
where
\begin{align}\label{eq:formula}
n_{K}=\frac{(-1)^{\dim\mathcal{V} _{i}^{K}}- \sum_{K<L}{n_{L}\, n(K,L)\, \left|  W(L)\right|  }}{\left|  W(K)\right|  }.
\end{align}
\end{definition}

\begin{lemma}\label{le:basic_coefficient}
If for $ (K_o) \in \Phi_0(G)$, one has $\text{coeff}^{L_o} (\deg_{\cV_i})$ is a leading coefficient of $\deg_{\cV_i}$, then $\dim (\cV_i ^{K_o})$ is odd and
\[
\text{coeff}^{K_o}(\deg_{\cV_i})=
\begin{cases}\label{eq:a0}
-1 & \text{if} \; |W(K_o)|=2,\\
-2 & \text{if}\; |W(K_o)|=1;\\
\end{cases}
\]
\end{lemma}
\begin{lemma}\label{le:involutive}
    For each  $\mathcal V_i$,  the corresponding basic degree $\deg_{\mathcal V_i} \in A(G)$ is an involution in the Burnside ring. It satisfies
    \[
    (\deg_{\cV_i})^2=\deg_{\cV_i} \cdot \deg_{\cV_i}=(G).
    \]\vskip0.3cm
\end{lemma}
\vskip.3cm

%---

\subsection{\texorpdfstring{$G$}{G}-Equivariant Gradient Degree}

%---
Let $V$ be an orthogonal $G$-representation. Denote by $C^{2}_{G}
(V,\mathbb{R})$ the space of $G$-invariant real $C^{2}$-functions on $V$. Let
$\varphi\in C^{2}_{G}(V,\mathbb{R})$ and $\Omega\subset V$ be an open bounded
invariant set such that $\nabla\varphi(x)\not =0$ for $x\in\partial\Omega$. In
such a case, the pair $(\nabla\varphi,\Omega)$ is called \textit{$G$-gradient
$\Omega$-admissible}. Denote by $\mathcal{M}^{G}_{\nabla}(V,V)$ the set of all
$G $-gradient $\Omega$-admissible pairs in $\mathcal{M}^{G}(V,V)$ and put
$\mathcal{M}^{G}_{\nabla}:=\bigcup_{V}\mathcal{M}^{G}_{\nabla}(V,V)$. 

\begin{theorem}
\label{thm:Ggrad-properties} There exists a unique map
$\nabla_{G}\mbox{\rm -}\deg:\mathcal{M}_{\nabla}^{G}\to U(G)$, which assigns
to every $(\nabla\varphi,\Omega)\in\mathcal{M}^{G}_{\nabla}$ an element
$\nabla_{G}\mbox{\rm -}\deg(\nabla\varphi,\Omega)\in U(G)$, called the
\textit{$G$-gradient degree} of $\nabla\varphi$ on $\Omega$,
\begin{equation}
\label{eq:grad-deg}\nabla_{G}\mbox{\rm -}\deg(\nabla\varphi,\Omega)=
\sum_{(H_{i})\in\Phi(G)}n_{H_{i}}(H_{i})= n_{H_{1}}(H_{1})+\dots
+n_{H_{m}}(H_{m}),
\end{equation}
satisfying the following properties:

\begin{enumerate}
%\item \renewcommand\labelenumi{\rm\bf($\nabla$\arabic{enumi})}

\item \textrm{\textbf{(Existence)}} If $\nabla_{G}\mbox{\rm -}\deg
(\nabla\varphi,\Omega)\not =0$, i.e., \eqref{eq:grad-deg} contains a non-zero
coefficient $n_{H_{i}}$, then $\exists_{x\in\Omega}$ such that $\nabla
\varphi(x)=0$ and $(G_{x})\geq(H_{i})$.

\item \textrm{\textbf{(Additivity)}} Let $\Omega_{1}$ and $\Omega_{2}$ be two
disjoint open $G$-invariant subsets of $\Omega$ such that $(\nabla
\varphi)^{-1}(0)\cap\Omega\subset\Omega_{1}\cup\Omega_{2}$. Then,
\[
\nabla_{G}\mbox{\rm -}\deg(\nabla\varphi,\Omega)= \nabla_{G}\mbox{\rm -}\deg
(\nabla\varphi,\Omega_{1})+ \nabla_{G}\mbox{\rm -}\deg(\nabla\varphi
,\Omega_{2}).
\]

\item \textrm{\textbf{(Homotopy)}} If $\nabla_{v}\Psi:[0,1]\times V\to V$ is a
$G$-gradient $\Omega$-admissible homotopy, then
\[
\nabla_{G}\mbox{\rm -}\deg(\nabla_{v}\Psi(t,\cdot),\Omega)=\mathrm{constant}.
\]

\item \textrm{\textbf{(Normalization)}} Let $\varphi\in C^{2}_{G}
(V,\mathbb{R})$ be a special $\Omega$-Morse function such that $(\nabla
\varphi)^{-1}(0)\cap\Omega=G(v_{0})$ and $G_{v_{0}}=H$. Then,
\[
\nabla_{G}\mbox{\rm -}\deg(\nabla\varphi,\Omega)= (-1)^{{m}^{-}(\nabla
^{2}\varphi(v_{0}))}\cdot(H),
\]
where ``$m^{-}(\cdot)$'' stands for the total dimension of eigenspaces for
negative eigenvalues of a (symmetric) matrix.

\item \textrm{\textbf{(Product)}} For all $(\nabla\varphi_{1},\Omega
_{1}),(\nabla\varphi_{2},\Omega_{2})\in\mathcal{M}^{G}_{\nabla}$,
\[
\nabla_{G}\mbox{\rm -}\deg(\nabla\varphi_{1}\times\nabla\varphi_{2},
\Omega_{1}\times\Omega_{2})= \nabla_{G}\mbox{\rm -}\deg(\nabla\varphi
_{1},\Omega_{1})\ast\nabla_{G}\mbox{\rm -}\deg(\nabla\varphi_{2},\Omega_{2}),
\]
where the multiplication `$\ast$' is taken in the Euler ring $U(G)$.

\item \textrm{\textbf{(Reduction Property)}} Let $V$ be an orthogonal
$G$-representation, $f:V\to V$ a $G$-gradient $\Omega$-admissible map, then
\begin{equation}
\label{eq:red-G}\pi_{0}\left[  \nabla_{G}\mbox{\rm -}\deg(f,\Omega)\right]
=G\mbox{\rm -}\deg(f,\Omega).
\end{equation}
where the ring homomorphism $\pi_{0}:U(G)\to A(G)$ is given by \eqref{eq:pi_0-homomorphism}.
\end{enumerate}
\end{theorem}
For other properties such as Functoriality, Hopf Property, Suspension, etc., one is referred to Section 6 of \cite{DK}.

\vskip0.3cm 

\subsection{Computations of the Gradient \texorpdfstring{$G$}{G}-Equivariant Degree}

%-
Similarly to the case of the Brouwer degree, the gradient equivariant degree
can be computed using standard linearization techniques. Therefore, it is
important to establish computational formulae for linear gradient operators.
\vskip.3cm Let $V$ be an orthogonal (finite-dimensional) $G$-representation
and suppose that $A:V\to V$ is a $G$-equivariant symmetric isomorphism of $V$,
i.e., $A:=\nabla\varphi$, where $\varphi(x)=\frac12 Ax\bullet x$. Consider the
$G$-isotypical decomposition of $V$
\[
V=\bigoplus_{i}V_{i},\quad V_{i}\;\;\mbox{modeled on}\;\mathcal{V}_{i}.
\]
We assume here that $\{\mathcal{V}_{i}\}_{i}$ is the complete list of
irreducible $G$-representations. \vskip.3cm Let $\sigma(A)$ denote the
spectrum of $A$ and $\sigma_{-}(A):=\{\alpha\in\sigma(A): \alpha<0\}$, and
let $E_{\mu}(A)$ stands for the eigenspace of $A$ corresponding to $\mu
\in\sigma(A)$. Put $\Omega:=\{x\in V: \|x\|<1\}$. Then, $A$ is $\Omega
$-admissibly homotopic (in the class of gradient maps) to a linear operator
$A_{o}:V\to V$ such that
\[
A_{o}(v):=
\begin{cases}
-v, & \mbox{if}\;v\in E_{\mu}(A),\;\mu\in\sigma_{-}(A),\\
v, & \mbox{if}\;v\in E_{\mu}(A),\;\mu\in\sigma(A)\setminus\sigma_{-}(A).
\end{cases}
\]
In other words, $A_{o}|_{E_{\mu}(A)}=-\text{\textrm{Id}}$ for $\mu\in
\sigma_{-}(A)$ and $A_{o}|_{E_{\mu}(A)}=\text{\textrm{Id}}$ for $\mu\in
\sigma(A)\setminus\sigma_{-}(A)$. Suppose that $\mu\in\sigma_{-}(A)$. Then,
denote by $m_{i}(\mu)$ the integer
\[
m_{i}(\mu):=\dim(E_{\mu}(A)\cap V_{i})/\dim\mathcal{V}_{i},
\]
which is called the \textit{$\mathcal{V}_{i}$-multiplicity} of $\mu$. Since
$\nabla_{G}\mbox{\rm -}\deg(\text{\textrm{Id}},\mathcal{V}_{i})=(G)$ is the
unit element in $U(G)$, we immediately obtain, by product property ($\nabla
$5), the following formula
\begin{equation}
\label{eq:grad-lin}\nabla_{G}\mbox{\rm -}\deg(A,\Omega)= \prod_{\mu\in
\sigma_{-}(A)}\prod_{i} \left[  \nabla_{G}\mbox{\rm -}\deg(-\text{\textrm{Id}
},B(\mathcal{V}_{i}))\right]  ^{m_{i}(\mu)},
\end{equation}
where $B(W)$ is the unit ball in $W$. \vskip.3cm

\begin{definition} \label{de:basicequi}
\textrm{\ Assume that $\mathcal{V}_{i}$ is an irreducible $G$-representation.
Then, the $G$-equivariant gradient degree:
\[
\nabla_{G}\text{-}\deg_{\mathcal{V}_{i}}:=\nabla_{G}\mbox{\rm -}\deg(-\text{\text{\rm Id}}
,B(\mathcal{V}_{i}))\in U(G)
\]
is called the \textit{gradient $G$-equivariant basic degree} for
$\mathcal{V}_{i}$. }
\end{definition}

\vskip.3cm

\begin{proposition}
\label{prop:invertible} The gradient $G$- equivariant basic degrees $\nabla_{G}\text{-}\deg_{\mathcal{V}_{i}}$ are invertible elements in $U(G)$.
\end{proposition}

\begin{proof}
Let $a:=\pi_{0}(\nabla_G\text{\textrm{-deg\,}}_{\cV_{i}})$, then $a^{2}=(G)$ in $A(G)$ (see Lemma \eqref{le:involutive}),
which implies that $(\nabla_{G}\text{\textrm{-deg\,}}_{\mathcal{V}_{i}})^{2}=(G)-\alpha$,
where for every $(H)\in\Phi_{0}(G)$ one has coeff$^{H}(\alpha)=0$. It is
sufficient to show that $(G)-\alpha$ is invertible in $U(G)$. Since (by
Proposition \ref{pro:nilp-elements}) for sufficiently large $n\in{\mathbb{N}}%
$, $\alpha^{n}=0$, one has
\[
\big((G)-\alpha)\sum_{n=0}^{\infty}\alpha^{n}=\sum_{n=0}^{\infty}\alpha
^{n}-\sum_{n=1}^{\infty}\alpha^{n}=(G),
\]
where $\alpha^{0}=(G)$
\end{proof}
\vskip.3cm \noindent\textbf{Degree on the Slice:} Let $\mathscr H$ be a
Hilbert $G$-representation. Suppose that the orbit $G(u_{o})$ of $u_{o}
\in\mathscr H$ is contained in a finite-dimensional $G$-invariant subspace, so
the $G$-action on that subspace is smooth and $G(u_{o})$ is a smooth
submanifold of $\mathscr H$. In such a case we call the orbit $G(u_{o})$
\textit{finite-dimensional}. Denote by $S_{o}\subset\mathscr H$ the slice to
the orbit $G(u_{o})$ at $u_{o}$. Denote by $V_{o}:=\tau_{u_{o}}G(u_{o})$ the
tangent space to $G(u_{o})$ at $u_{o}$. Then $S_{o}=V_{o}^{\perp}$ and $S_{o}
$ is a smooth Hilbert $G_{u_{o}}$-representation.

Then we have (cf. \cite{BeKr}):

\begin{theorem}
\textrm{(Slice Principle)} \label{thm:SCP} Let $\mathscr{H}$ be a Hilbert
$G$-representation, $\Omega$ an open $G$-invariant subset in $\mathscr H$, and
$\varphi:\Omega\rightarrow\mathbb{R}$ a continuously differentiable
$G$-invariant functional such that $\nabla\varphi$ is a completely continuous
field. Suppose that $u_{o}\in\Omega$ and $G(u_{o})$ is an finite-dimensional
isolated critical orbit of $\varphi$ with $S_{o}$ being the slice to the orbit
$G(u_{o})$ at $u_{o}$, and $\mathcal{U}$ an isolated tubular neighborhood of
$G(u_{o})$. Put $\varphi_{o}:S_{o}\rightarrow\mathbb{R}$ by $\varphi
_{o}(v):=\varphi(u_{o}+v)$, $v\in S_{o}$. Then
\begin{equation}
\nabla_{G}\text{\text{\rm-deg\,}}(\nabla\varphi,\mathcal{U})=\Theta
(\nabla_{G_{u_{o}}}\text{\text{\rm-deg\,}}(\nabla\varphi_{o},\mathcal{U}\cap
S_{o})), \label{eq:SDP}%
\end{equation}
where $\Theta:U(G_{u_{o}})\rightarrow U(G)$ is homomorphism defined on
generators $\Theta(H)=(H)$, $(H)\in\Phi(G_{u_{o}})$.
\end{theorem}
\section{General Algorithm for Maximal Orbit Type computations}\label{sec:E}
Although maximal orbit types for bifurcation invariants in our case can be directly identified through computation, we briefly outline below the general algorithm for maximal orbit types appearing in two basic degree product. Given a maximal orbit type $(H)$ associated with a particular irreducible representation, we now analyze how the coefficient of $(H)$ behaves in the bifurcation invariant $\omega_G(\alpha_{j_o}).$\vskip0.3cm
\noindent Put 
\[
\mathscr D^j:=\Big\{(\Sigma):(\Sigma) \text{ is maximal in } \nabla
\text{\textrm{-deg}}_{\mathcal{W}_{j}}\Big\}.
\]
For two basic degrees $\nabla
\text{\textrm{-deg}}_{\mathcal{W}_{j_o}}$ and $\nabla
\text{\textrm{-deg}}_{\mathcal{W}_{\tilde{j}_o}},$ let $(H)\in \mathscr D^{j_o},\; (K) \in \mathscr D^{\tilde{j}_o}.$ We may write
\[
\nabla
\text{\textrm{-deg}}_{\mathcal{W}_{j_o}}=(S_5)+n_H(H)+\cdots,\qquad \nabla
\text{\textrm{-deg}}_{\mathcal{W}_{\tilde{j}_o}}=(S_5)+n_K(K)+\cdots,
\]
where dots denote terms corresponding to submaximal orbit types. By Lemma \eqref{le:basic_coefficient},
\[
n_H=\begin{cases}
    -1,\;\;\; |W(H)|=2\\
    -2,\;\;\; |W(H)|=1,
\end{cases}
\]
where $n_K$ satisfies the analogous property. We next have the following several cases with respect to our situation \eqref{eqn:basicdegree}.
\begin{enumerate}
\item[(i)]Suppose $H=K.$ Then 
\[
\nabla
\text{\textrm{-deg}}_{\mathcal{W}_{j_o}}*\nabla
\text{\textrm{-deg}}_{\mathcal{W}_{\tilde{j_o}}}=(S_5)+\Big(2n_H+n_H^2|W(H)|\Big)(H)+\cdots,
\]
and using $2n_H+n_H^2|W(H)|=0$ (see \cite{liu2025nonlinear} for more details), we conclude that the maximal term $(H)$ cancels in the product.
\item[(ii)] Suppose $H\neq K$ and $H\cap K=Q,$ where $Q$ can be detected by both degrees. Then 
\[
\nabla
\text{\textrm{-deg}}_{\mathcal{W}_{j_o}}*\nabla
\text{\textrm{-deg}}_{\mathcal{W}_{\tilde{j_o}}}=(S_5)+n_H(H)+n_K(K)+n_Hn_Kn_Q(Q)+\cdots,
\]
and
\[
n_Q=\frac{n(Q,K)|W(K)|n(Q,H)|W(H)|-\sum_{(\widetilde{Q})>(Q)}n(Q,\widetilde
{Q})n_{\widetilde{Q}}|W(\widetilde{Q})|}{|W(Q)|} 
\]
ensures that both maximal types $(H)$ and $(K)$ persist in the product, with their coefficients also preserved. A special subcase arises when $H\subset K:$  
\[
\nabla
\text{\textrm{-deg}}_{\mathcal{W}_{j_o}}*\nabla
\text{\textrm{-deg}}_{\mathcal{W}_{\tilde{j_o}}}=(S_5)+n_H(H)+n_K(K)+n_Hn_Kn_{((H)*(K))}(H)+\cdots,
\]
where
\[
n_{((H)*(K))}=\frac{n(H,K)|W(K)|n(H,H)|W(H)|}{|W(H)|}=n(H,K)|W(K)||W(H)|.
\]
Since both $(H)$ and $(K)$ are maximal and $|W(K)|$ divides $|W(H)|,$ necessarily, 
\[
|W(K)|=1,|W(H)|=2 \text{  and  } n_K=-2, n_H=-1.
\]
Therefore,
\[
\nabla
\text{\textrm{-deg}}_{\mathcal{W}_{j_o}}*\nabla
\text{\textrm{-deg}}_{\mathcal{W}_{\tilde{j_o}}}=(S_5)+\Big(-1+4n(H,K)\Big)(H)-2(K)+\cdots,
\]
\end{enumerate}

\newpage

\bibliographystyle{plain}
\bibliography{mybib_minima}

\end{document}